\documentclass[
a4paper,
11pt]{amsart}

\usepackage{amsmath,amsfonts,amsthm,mathrsfs}
\newtheorem{theorem}{Theorem}[section]
\newtheorem{lemma}[theorem]{Lemma}
\newtheorem{corollary}[theorem]{Corollary}
\newtheorem{defn}[theorem]{Definition}
\newtheorem{remark}[theorem]{Remark}
\usepackage[a4paper]{geometry}
\usepackage{xcolor}
\usepackage{enumerate}
\newcommand{\ve}{\varepsilon}

\newcommand\cA{{\mathcal A}}
\newcommand\cB{{\mathcal B}}
\newcommand\cL{{\mathcal L}}

\newcommand{\BB}{\mathcal{B}}

\newcommand{\FF}{\mathcal{F}}

\newcommand{\HH}{\mathcal{H}}
\newcommand{\JJ}{\mathcal{J}}

\newcommand{\LL}{\mathcal{L}}

\newcommand{\RR}{\mathcal{R}}

\newcommand{\ZZ}{\mathcal{Z}}
\def\D{\mathrm {d}}

\newcommand{\sgn}{{\rm sgn}}

\DeclareMathOperator{\supp}{supp}

\newcommand{\real}{\mathbb{R}}
\newcommand{\complex}{\mathbb{C}}

\newcommand\bR{{\mathbb R}}
\newcommand\bZ{{\mathbb Z}}

\DeclareMathOperator{\sign}{sign}

\begin{document}

\title[{fractional susceptibility function of piecewise expanding maps}]{On the fractional susceptibility function of piecewise expanding maps}
\author[M. Aspenberg, V. Baladi, J. Lepp\"anen, and T. Persson]{Magnus 
Aspenberg\textsuperscript{$\dagger$}, Viviane Baladi\textsuperscript{$\dagger\dagger$} , Juho Lepp\"anen\textsuperscript{$\dagger\dagger$}, \\ and Tomas Persson\textsuperscript{$\dagger$}}
\address{$\dagger$Centre for Mathematical Sciences, Lund University, 
Box 118, 221 00 Lund, Sweden}
\email{magnusa@maths.lth.se } 
\email{tomasp@maths.lth.se}

\address{$\dagger\dagger$Laboratoire de Probabilit\'es, Statistique et Mod\'elisation (LPSM), CNRS, \phantom{SORBONNE}
\phantom{$\dagger\dagger$} Sorbonne Universit\'e, Universit\'e
de Paris, 4 Place Jussieu, 75005 Paris, France}
\email{baladi@lpsm.paris}

\email{lepp.juho@gmail.com}

\date{\today}
\thanks{We are grateful to Daniel Smania, some of  whose ideas in the collaboration
	\cite{BS3} were very useful here.
Part of  this work was carried out at the Centre for Mathematical Sciences, Lund University,
during VB's  Knut and Alice Wallenberg Guest Professorship. VB's and JL's research is supported
by the European Research Council (ERC) under the European Union's Horizon 2020 research and innovation programme (grant agreement No 787304).}

\begin{abstract}
We associate to a perturbation $(f_t)$ of a (stably
mixing) piecewise expanding unimodal map $f_0$
a two-variable fractional susceptibility function $\Psi_\phi(\eta, z)$,
depending also on a bounded observable $\phi$.
For fixed $\eta \in (0,1)$, we show that the function  $\Psi_\phi(\eta, z)$
is holomorphic in a disc $D_\eta\subset \complex$
centered at zero of radius $>1$, and that $\Psi_\phi(\eta, 1)$
is the Marchaud fractional derivative of order $\eta$  of the function
$t\mapsto \RR_\phi(t):=\int \phi(x)\, d\mu_t$, at $t=0$, where $\mu_t$ is the unique absolutely
continuous invariant probability measure of $f_t$. 
In addition, we show that $\Psi_\phi(\eta, z)$ admits a holomorphic extension to
the domain $\{\, (\eta, z) \in \complex^2\mid 0<\Re \eta <1, \, z \in D_\eta \,\}$.
Finally, if the perturbation $(f_t)$ is horizontal, we prove that
 $\lim_{\eta \in (0,1), \eta \to 1}\Psi_\phi(\eta, 1)=\partial_t
\RR_\phi(t)|_{t=0}$.
\end{abstract}
\maketitle

\section{Introduction}

\subsection{Linear response and violation thereof}
\label{forapp}

Response theory describes how the ``physical measure'' $\mu_t$
of a dynamical system $f_0$ responds to perturbations $t\mapsto f_t$, given a  class
of observables (test functions)
$\phi$. 
Classically, one studies \emph{linear response,}
where the goal is to express the derivative of
$\RR_\phi(t ):=\int \phi \, \D \mu_t$, for a fixed $\phi$ in a suitable  class, 
in terms of  ($\phi$ together with) $f_0$, the vector field $v_0:= \partial_t f_t|_{t=0}$, and the measure $\mu_0$.

Linear response was first investigated \cite{Co, KKPW,Ru,GL} for
smooth hyperbolic dynamics (Anosov or Axiom~A).  In the smooth mixing
hyperbolic case, the physical measure $\mu_t$ is the SRB measure
\cite{LSY} and corresponds to the fixed point of a transfer operator
$\LL_t$ (whose dual preserves Lebesgue measure) on a suitable Banach
space $\BB$ (see e.g.\ \cite{BaladiZeta}). This fixed point is a
simple isolated eigenvalue in the spectrum of $\LL_t$, and linear
response can be proved via perturbation theory for simple eigenvalues
(see e.g.\ \cite[\S2.5 and \S5.3]{BaladiZeta}).

To avoid technicalities, we write the key formulas in the easy case
of smooth expanding circle maps (see e.g.\  \cite{Seoul} for more
details): Then, the physical measure is the unique absolutely continuous invariant
probability measure, that is, $\mu_t =\rho_t \, \D x$, with
$\rho_t$ smooth. The operator $\LL_t$ acts on smooth functions, and 
we have the trivial but key identity
\begin{equation}\label{key}
\rho_t-\rho_0= (I-\LL_t)^{-1}(\LL_t -\LL_0 )\rho_0 \, .
\end{equation} 
Next, assuming that $v_0=X_0\circ f_0$,
it is not hard to show that
\[
\lim_{t \to 0} \frac{(\LL_t -\LL_0 )\rho_0}{t} = -(X_0 \rho_0)'\, .
\]
Then, since $\int \phi \LL^k_0(\psi)\,  \D x= \int (\phi\circ f_0^k) \psi\,  \D x$, we have,\footnote{Note that $\int_{S^1} (X_{0}\rho_{0})' \, \D x=0$ because the circle has no boundary.} for continuous $\phi$ (say),
\[
 \int \phi \cdot 
 (1-\LL_{0})^{-1} (X_{0}\rho_{0})'  \, \D x
=\sum_{k=0}^\infty \int  \phi \cdot 
  \LL^k_{0} (X_{0}\rho_{0})' \,  \D x=
\sum_{k=0}^\infty \int  (\phi \circ f_0^k)  (X_{0}\rho_{0})'\, 
\D x \, .
\]
Finally, using \eqref{key}, we get the \emph{fluctuation--dissipation
  formula} (expressing the derivative as a Green--Kubo \cite{K} sum of
decorrelations)
\begin{align}
\label{FDT}
\partial_t \RR_\phi(t)|_{t=0} =\partial_t (\int \phi \, \D \mu_t )|_{t=0}
&=-\sum_{k=0}^\infty \int  (\phi \circ f_0^k)  (X_{0}\rho_{0})'\, 
\D x \, ,
\end{align}
and, if $\phi$ is differentiable,
integrating by parts, we get the \emph{linear response formula}
\begin{align}
\partial_t \RR_\phi(t)|_{t=0}=\partial_t (\int \phi \, \D \mu_t )|_{t=0}
\label{LRF} & =
 \sum_{k=0}^\infty \int    (\phi \circ f_{0}^k)' 
\cdot 
 (X_{0} \, \rho_{0}) \, \D x\, .
\end{align}
Exponential decay of the series in the right-hand side of \eqref{LRF}
is not as apparent as in \eqref{FDT}, since the derivative $(\phi
\circ f_{0}^k)'(x)=\phi'(f_{0}^k(x)) \cdot(f_0^k)'(x)$ grows
exponentially. However, the presence of this derivative should be
expected\footnote{See e.g.\ the heuristic argument leading to
  \cite[(2)]{Ru2}.} when describing response, and Ruelle \cite{Ru1}
pointed out that it was meaningful to view the linear response formula
\eqref{LRF} as the value at $z=1$ of a natural power series, the
\emph{susceptibility function} of $f_0$ and $v_0=X_0\circ f_0$.  In
the present smooth one-dimensional expanding case, setting
$z=e^{i\omega}$, the susceptibility function $\Psi_\phi(z)$ is the
Fourier transform of the \emph{response function} $k\mapsto \int X_0
(\phi \circ f_0^k)' \rho_0 \, \D x$, that is:
\begin{equation}\label{SF}
\Psi_\phi(z):=
\sum_{k=0}^\infty z^k  \int    (\phi\circ  f_0^k)' \, (X_0 \rho_0) \, \D x \, .
\end{equation}
We have, integrating by parts,
\[
\Psi_\phi(z)= - \sum_{k=0}^\infty z^k  \int    (\phi \circ f_0^ k) \, (X_0 \rho_0)' \, \D x\, ,
\]
so that exponential mixing implies
that the susceptibility function $\Psi_\phi(z)$ is holomorphic in a disc of radius larger
than one, using that $X_0 \rho_0$ is smooth.

\medskip

In situations when the map $t\mapsto \RR_\phi(t)$ is not differentiable (either due to bifurcations in the
dynamics \cite{Ma, Baladi2007,BS, BBS}, or to singularities 
\cite{BKL, Po} of the test function  $\phi$),
it is natural to consider \emph{fractional response,} i.e., to investigate weaker moduli of continuity
of this map.
The simplest situation where linear response breaks down is that of piecewise
expanding unimodal maps. In this case, the transfer operator has
a spectral gap when acting on $BV$. However,  the derivative $\rho_0'$ of
the invariant density $\rho_0$ involves a sum of Dirac masses,
and thus does not belong to $BV$, or to a space on which decorrelations
are summable. Keller  \cite{Ke} showed in 1982 that
$|\rho_t-\rho_0|_{L^1}= O(|t||\log |t||)$.
 Examples of families $(f_t)$ and smooth functions $\phi$ such that 
 $|\RR_\phi(t)-\RR_\phi(0)|\ge |t||\log |t||$ were described in \cite{Ma} and
 \cite[Theorem~6.1]{Baladi2007}.
 (See also the previous work of Ershov \cite{Er}.)
Baladi and Smania \cite{BS} showed that if the family $f_t$ is \emph{tangential}
to\footnote{We refer to  the beginning
of \S\ref{proofcor} for a definition of tangentiality, which
is equivalent to the 
\emph{horizontality} condition  \eqref{horr} on $v_0$.}
the topological class of $f_0$, then for any continuous
function $\phi$, the map $\RR_\phi(t)$
is differentiable at $t=0$. They also showed that when
horizontality does not hold, then there exist
smooth observables $\phi$ such that $\RR_\phi(t)$
is \emph{not} Lipschitz at $t=0$.
More recently, de Lima and Smania  \cite{dLS} showed  central limit theorems which imply that
for a generic map $f_0$, if $v_0=\partial_t f_t|_{t=0}$ is not horizontal,
then for a generic  $\phi$, the map $t \mapsto \int \phi \, \D \mu_t$
cannot be Lipschitz on a set of parameters $t$ of positive Lebesgue measure.
(See also \cite{C} for a related result.)

In the piecewise expanding unimodal case, the susceptibility function 
$\Psi_\phi(z)$ can also be defined by the formal power series \eqref{SF}, and it
has been studied in \cite{Baladi2007, BS, BMS}.
In  particular \cite{BMS}, if the postcritical orbit
is dense (a generic condition) then  $\Psi_\phi(z)$ has a strong natural boundary
on the unit circle, while
if the perturbation $v_0=X_0\circ f$ is horizontal
and, in addition (a generic condition) the postcritical orbit
is Birkhoff typical, then
 the nontangential limit of $\Psi_\phi(z)$ as $z$ tends to $1$ 
 coincides with $\partial_t\RR_\phi(t)|_{t=0}$.

\smallskip
In the present paper, we introduce and study a two-variable fractional susceptibility function
$\Psi_\phi(\eta, z)$, for $\Re \eta \in (0,1)$,  in
the setting of piecewise expanding unimodal maps. The initial motivation
for this work comes from the following paradox:

In 2005, Ruelle and Jiang considered  \cite{Ru0, JR}
 finite Misiurewicz--Thurston (MT)
parameters $t_0$ in the quadratic family $f_t(x)=t-x^2$.
By definition of MT, there exists a repelling periodic
point $x_0$ and $M\ge 2$ such that $f_{t_0}^M(0)=x_0$.
Such  parameters $t_0$ are  Collet--Eckmann, so $f_{t_0}$ admits
a unique absolutely continuous invariant probability measure $\mu_{t_0}$.
The map $f_{t_0}$ also has a finite Markov partition, which simplifies the analysis.
Ruelle and Jiang proved that,
for any $C^1$ observable $\phi$, the susceptibility
function  $\Psi_\phi(z)$ defined by \eqref{SF} is meromorphic in the whole complex plane, and
that $z=1$ is not a pole. Since $\Psi_\phi(1)$ is the natural candidate
for the derivative of $\RR_\phi(t)$,  this raised the hope that
there could exist a ``large'' subset $\Omega$ 
of Collet--Eckmann parameters, containing $t_0$, and  such that $t\mapsto \RR_\phi(t)$
would be differentiable in the sense of Whitney on $\Omega$ at $t_0$.
However, Baladi, Benedicks, and Schnellmann  \cite{BBS}
later showed that for any mixing   
(non horizontal\footnote{All MT parameters are  non horizontal by \cite{Avila2002} or \cite{Levin2002}.})
MT parameter $t_0$, there exist a $C^\infty$ observable $\phi$, a sequence
$t_n \to t_0$ of Collet--Eckmann parameters (with bounded constants), and
 $C>1$ 
such that
\begin{equation}\label{breakt}
\frac{ \sqrt{|t_n-t_0|}}
{C} \le |\RR_\phi(t_n)-\RR_\phi(t_0)| \le C \sqrt{|t_n-t_0|}
\, , \quad  \forall n \, .
\end{equation}
It is not known whether $t_0$ is a Lebesgue density point
in the set of $t_n$ such that \eqref{breakt} holds.
Also, the analogue of the de Lima--Smania \cite{dLS}
central limit theorems  is not known in this setting. However, we expect
that a similar CLT holds, and in particular
that  $\Psi_\phi(1)$ \emph{cannot} be interpreted as the derivative
of $\RR_\phi(t)$ in the sense of Whitney on a set
of parameters containing $t_0$ as a Lebesgue density point.
The fact that $\Psi_\phi(z)$ is holomorphic at $z=1$ can thus be viewed as a
paradox. 

\medskip

Baladi and Smania  \cite{BS3}
very recently introduced two-variable fractional susceptibility
functions $\Psi_\phi(\eta,z)$ for the quadratic family, with the goal of
resolving this paradox. 
The piecewise expanding
case is a toy model for the quadratic setting, and indeed, several ideas previously
developed for piecewise expanding
families \cite{BS} were crucial to obtain the breakthrough result \eqref{breakt}
of \cite{BBS} for the
quadratic family. Although  we believe it is interesting in its own right,
the analysis carried out in the present paper for this
toy model  can also be viewed as  a ``proof of concept,'' establishing
the feasibility of the fractional susceptibility function approach. 

\subsection{Informal statement of the results}
\label{1.2}
We next describe briefly our main results.
It is well-known that there is no canonical notion of a fractional derivative, see \cite{sam}
and \cite{MR} for a presentation of the theory.
We use here the Marchaud fractional derivative, defined 
for suitable functions $g$ by
\begin{align*}
  (M^\eta_t g (t))|_{t=0} &
 = \frac{\eta}{2 \Gamma(1-\eta)}
  \int_{-\infty}^\infty \frac{g (t) - g(0)}{|t|^{1+\eta}} \sign(t)
    \, \D t \, ,
\end{align*}
where $\Gamma$ denotes Euler's Gamma function. Our motivation for
using this particular fractional derivative is threefold: First,
$M^\eta g$ can be well-defined even if $g$ does not decay at infinity.
Second, the Marchaud fractional derivative of a constant function
vanishes, while this is not the case for other fractional derivatives,
in particular for the Bessel potential derivative defined by
$\FF^{-1}(1+|\xi|^2)^{\eta/2}\FF g$, where $\FF$ is the Fourier
transform.  Finally, the expression of Marchaud derivatives in terms
of differences $g(t)-g(-t)=g(t)-g(0) + g(0)- g(-t)$ is convenient in
view of \eqref{key}. Nevertheless, we expect that fractional
susceptibility functions defined via (e.g.) Bessel, or
Riemann--Liouville fractional derivatives would enjoy similar
properties as those we establish here using the Marchaud derivative.

We define  for $\eta\in (0,1)$
the \emph{fractional susceptibility function} of the perturbation $(f_t)$ and  
the observable $\phi \in L^\infty$, 
to be the formal power series
\begin{align}\label{1}
\Psi_\phi (\eta, z) 
&:= 
 \sum_{k=0}^\infty  z^k \int_I \frac{ \eta}{2 \Gamma(1-\eta)}
\int_{-\infty}^\infty (\phi \circ f_t^k) (x) \cdot 
\frac{\bigl ((\LL_{ t}-\LL_0 )  \rho_0\bigr )(x) }{|t|^{1+\eta}} \sign(t) \, \D t \,
\D x\, ,
\end{align}
and the \emph{frozen
fractional susceptibility function} of $(f_t)$ and $\phi$ to be the formal power series
\begin{align}\label{2}
  \Psi^\mathrm{fr}_\phi(\eta,z)
&:= \sum_{k=0}^\infty  z^k \int_I  
\frac{ \eta}{2 \Gamma(1-\eta)} (\phi \circ f_0^k) (x) 
\int_{-\infty}^\infty 
\frac{\bigl ((\LL_{ t} -\LL_0)  \rho_0\bigr )(x)}{|t|^{1+\eta}} \sign(t) \, \D t \,
\D x\, .
\end{align}
In both susceptibility functions, $\LL_t$ is the transfer operator 
$\LL_t \varphi(x)=\sum_{f_t(y)=x} \frac{\varphi(y)}{|f'_t(y)|}$.

\medskip
Our first main result (Theorem~\ref{thm:susc_holom}) 
says that if $\phi$ is bounded, then for any $\eta \in (0,1)$ the improper integrals in 
the two above formal power series are convergent, and that 
$\Psi_\phi(\eta,z)$ and  $\Psi_\phi^\mathrm{fr}(\eta,z)$ are holomorphic functions of $z$
in a disc $D_\eta$  of radius (which may depend on $f_0$) larger than one.  In addition, 
$$ \Psi^\mathrm{fr}_\phi(\eta,z)
=  \int_I  \phi \cdot (I-z\LL_0)^{-1}
M^\eta_t  \bigl (\LL_t \rho_0 (x) \bigr )|_{t=0}\, 
\D x= \sum_k z^k
 \int_I  (\phi \circ f_0^k)
M^\eta_t  \bigl (\LL_t \rho_0 (x) \bigr )|_{t=0}\, 
\D x\, ,
$$
and
$$\Psi_\phi(\eta,1)=M^\eta_t \RR_\phi(t)|_{t=0}\, .
$$

\begin{remark}Our lower bound
for the radius of $D_\eta$ tends to $1$ as $\eta \to 1$.  If the critical
point of $f_0$ is preperiodic,
then we expect $\Psi_\phi(\eta, z)$ to be holomorphic in
a disc of radius strictly larger than $1$ and meromorphic in the entire complex plane
for all $0\le \eta \le 1$. However, we believe that, generically,  the radius of convergence
of $\Psi_\phi(\eta, z)$ should tend to $1$
as $\eta \to 1$.
\end{remark}

As a corollary of our main theorem and the results
of \cite{BS}, we obtain (Corollary~\ref{lecor}) that,
if $v_0$ is horizontal  and $\phi$ is continuous, then 
$$\lim_{\eta \to 1}\Psi_\phi(\eta, 1)=\lim_{\eta \to 1}\Psi^\mathrm{fr}_\phi(\eta, 1)
=\Psi_\phi(1)=\partial_t \RR_\phi(t)|_{t=0}\, .
$$

Our second result, Theorem~\ref{newth}, is about more general fractional moduli of continuity:
We consider there weighted Marchaud derivatives,
replacing $|t|^{-1-\eta}$ by $(\log |t|)^{-\beta} |t|^{-1-\eta}$.
Applying Theorem~\ref{newth} to   $\beta >1$ and $\eta=1$ gives a modulus
of continuity $|t| |\log |t||^\beta$, almost reproducing the  $|t| |\log |t||$
estimates  from \cite{Ke, Ma, Baladi2007}.
Applying Theorem~\ref{newth} to   $\beta <0$, we show
(Corollary~\ref{cor:realanalytic})
that
$\Psi_\phi(\eta,z)$ and $\Psi^\mathrm{fr}_\phi(\eta,z)$
are holomorphic in the domain $\{\, (\eta, z)\in \complex^2\mid 0<\Re \eta<1\, ,
z \in D_\eta \,\}$.
\smallskip

The fractional susceptibility functions in \eqref{1} and \eqref{2} are  of ``fluctuation--dissipation'' type. 
It is tempting to consider the  \emph{response fractional susceptibility
function}\footnote{\label{samm}The equality \eqref{3} follows by integration by parts for the Marchaud derivative
\cite[(6.27)]{sam} for $\phi \in C^1$.}
\begin{align}
  \Psi^\mathrm{rsp}_\phi(\eta,z)
&:= -\sum_{k=0}^\infty  z^k \int_I  (\phi \circ f_0^k)\cdot
M^\eta_x  \bigl (X_0 \rho_0  \bigr )\,
\D x
\nonumber \\
&
\label{3}=
\sum_{k=0}^\infty  z^k \int_I  (M^\eta_x (\phi \circ f_0^k))
 \cdot X_0 \rho_0 \,
\D x ,
\end{align}
where the Marchaud derivative is now taken with respect to
$x$.  
The  arguments in the present
paper easily show that  for all $\eta\in (0,1)$ the function
$\Psi^\mathrm{rsp}_\phi(\eta,z)$ is holomorphic
in a disc of radius larger than $1$ for  all bounded $\phi$, and that
in the horizontal case we have $\lim_{\eta \to 1}\Psi^\mathrm{rsp}_\phi(\eta,1)=\partial_t \RR_\phi|_{t=0}$. 
The function $\Psi_\phi^\mathrm{rsp}(\eta,z)$  is at first sight the most seductive fractional
susceptibility function. However,
  $\Psi^\mathrm{rsp}_\phi(\eta,1)$
has no reason to coincide  in general with any fractional derivative
of order $\eta$ of $\RR_\phi(t)$  (except in the limit $\eta\to 1$,
even in the linear examples studied in \S\ref{lin}). In this respect,  $\Psi^\mathrm{rsp}_\phi(\eta,z)$
is not better than  the frozen susceptibility
function $\Psi^\mathrm{fr}_\phi(\eta,z)$.
Keeping also in mind that the goal of this fractional approach is to resolve the paradox described above for
the quadratic family,  we focus on the definitions \eqref{1} and \eqref{2} in the
present paper, referring to \cite{BS3} for more on the fractional response susceptibility function.

Observe also that fractional derivatives do not enjoy a  Leibniz formula
with finitely many terms for the 
derivative of a product, so we cannot expect fractional susceptibility functions
to be  as well-behaved as ordinary ones.

\medskip
We next make a few comments about our method of proof.
We shall consider transfer operators acting on  Sobolev spaces $\HH^{\tau,p}$
with $p>1$, close to $1$, and  $0<\tau <1/p$. These spaces give us more flexibility than
the $BV$ spaces classically used for piecewise expanding interval maps. We exploit
the bounds of   Thomine \cite{Thomine} for  the essential spectral radius
of $\LL_t$ on such spaces $\HH^{\tau,p}$.
In particular,  we have $\rho_0 \in \HH^{\tau,p}$
(see the beginning of \S\ref{standard}), so that, for each $\eta <1$,
the function  $M^\eta_x \rho_0$ belongs to a Sobolev space on which the transfer
operator $\LL_t$ associated to $f_t$ has a spectral gap (see \eqref{oldstuff}).
In order to apply the stable exponential decorrelation result of
Keller--Liverani \cite{KellerLiverani}, we need  Lasota--Yorke bounds
which are uniform in $t$.
Such bounds were known for the BV norm, but we have to carry out the corresponding estimates
for the Sobolev spaces.

The paper is organised as follows: Section~\ref{two2} starts with
definitions and formal statements. In \S\ref{frfr}, we state precisely  Theorem~\ref{thm:susc_holom} on fractional susceptibility
functions, followed by  its consequence (Corollary~\ref{lecor}) in the horizontal case.
In \S \ref{KellerlogS} we state Theorem~\ref{newth} about other moduli of continuity and Corollary~\ref{cor:realanalytic} about holomorphic extensions of the fractional
susceptibility functions.
In \S\ref{lin}, we discuss two linear examples.
Section~\ref{tools} introduces the key tools:
In \S\ref{standard}, we recall  properties of the transfer operator acting on
Sobolev spaces, from the work of Thomine \cite{Thomine}.
In \S\ref{stabm}, we discuss stability of mixing and mixing rates for good families, 
recalling in particular the results of Keller and Liverani \cite{KellerLiverani} that
we shall use, and stating the relevant technical lemmas (Lemma~\ref{lem:Lt-L0} and
Lemma~\ref{lem:lasota_yorke}). Subsection~\ref{theend} contains the
 proof of the perturbation Lemma~\ref{lem:Lt-L0} for Sobolev spaces which is needed to
 apply results of Keller and Liverani. In Section~\ref{pr},
we prove Theorems~\ref{thm:susc_holom} and \ref{newth}, as well
as Corollaries~\ref{lecor} and~\ref{cor:realanalytic}, using the techniques presented in Section~\ref{tools}.
Appendix~\ref{ll} contains the simple proof  that the limit
of the Marchaud derivatives as $\eta \to 1$ is the ordinary derivative. 
In Appendix~\ref{AA}, we show the uniform Lasota--Yorke estimates (Lemma~\ref{lem:lasota_yorke})
needed to apply the results of Keller and Liverani to Sobolev spaces.

Throughout, we shall use the notation $C$ for a finite positive constant which can vary
from place to place.


\section{Definitions and formal statement of results}
\label{two2}

\subsection{Fractional susceptibility functions for piecewise
expanding interval maps}
\label{frfr}

Let $I$ be the compact interval  $[-1,1]$ and let $r\in\{2,3\}$. We say that a continuous map $f \colon I \to I$ is
a piecewise $C^r$ unimodal map if there exists $-1 < c < 1$ such that
$f$ is increasing on $I_+ = [-1,c]$, decreasing on\footnote{The results of
 Thomine \cite{Thomine} hold for
piecewise $C^{1+\eta_0}$ maps, up to
replacing the condition $\tau <1/p$ there by $\tau < \max(\eta_0,1/p)$.
It would be interesting to adapt our results
to this setting, for $0<\eta < \eta_0$.}
  $I_{-} = [c,1]$, and   $f | _{I_{\sigma}}$ extends as a
  $C^r$ map denoted $f_\sigma$ to a neighbourhood $\tilde I_\sigma$ of $I_\sigma$. If, in addition, 
  $\lambda(f):=\inf_{\sigma=\pm } |f_\sigma'(x)| > 1$,
   we say that $f \colon I \to I$ is
a piecewise $C^r$ expanding unimodal map, and we define
\begin{equation}\label{ln}
\lambda_n(f):=\inf_{\vec \sigma\in \{\pm \}^n}\inf_x |(f^n_{\vec \sigma})'(x)|\, ,\,\,\,\,\, \,\,
\Lambda(f)= \biggl ( \lim_{n \to \infty} \lambda_n(f)^{-1/n}\biggr ) ^{-1} , 
\end{equation}
for all $n\ge 1$, where $f^n_{\vec \sigma}$ is the composition of the
maps $f_{\sigma_i}$, $i=1, \ldots, n$, and the points $x$ in
\eqref{ln} are restricted to those $x \in \tilde I_{\sigma_1} $ for
which the composition exists.  Following \cite{BS}, we say that a
piecewise $C^r$ expanding unimodal map $f$ is
\emph{good}\footnote{\label{goodness}Goodness will ensure the uniform
  Lasota--Yorke Lemma~\ref{lem:lasota_yorke} and thus stable
  mixing. It is not the weakest possible condition, see
  e.g.\ \cite[Example~5.6]{EG} but some assumption is needed
  \cite[Remark~15]{Ke}.  See also \cite[Proposition~2.1]{BS09}.}  if,
either $c$ is not periodic for $f$, or $|(f^{P_f})'(c)|>2$ where $P_f
\ge 2$ is the minimal period of $c$.  Finally, we say that a piecewise
$C^r$ expanding unimodal map $f$ is \emph{mixing} if $f$ is
topologically mixing on $[f^2(c), f(c)]$.

\begin{defn}[Perturbation $(f_t)$ of  a piecewise $C^r$ expanding unimodal map]\label{defL}
  Let $r\in\{ 2, 3\}$ and let $f$ be a piecewise $C^r$ expanding unimodal map. 
  Given $\ve > 0$ and  a family $(f_t)_{|t| \le \ve}$ of piecewise
  $C^r$  unimodal maps (for fixed $\tilde I_\pm$) is called a $C^{r}$ perturbation of $f_0=f$ if
 the $C^r$ norm of the extension $f_{t, \sigma}$ of $f_t$ to $\tilde I_\pm$ is uniformly bounded,
 with
$\|(f - f_t) |_{\tilde  I_{\sigma}} \|_{C^{r-1}} = O(|t|)\, 
\mbox { as } t \to 0$ for  $\sigma=\pm$. 
 
\noindent If all  $f_t$ are topologically
conjugated, or if $c$ is not periodic for $f$, set
$
\Lambda:=\inf_{|t|<\ve}\ \Lambda(f_t) 
$,
otherwise, if $f$ is good, with $c$ periodic of minimal
period $P_f\ge 2$, set
\begin{equation}\label{defLe}
\Lambda^{-1} :=\sup_{|t|<\ve}\, 
	 \lim_{k \to \infty} \biggl ( \frac{ \lambda_{k P_f}(f_t) }{2^k}
\biggr )^{-\frac{1}{kP_f}} \, ;
\end{equation}
in all these cases, up to taking a smaller $\ve$, we may and shall assume that  $\Lambda>1$. 
\end{defn}

Any piecewise $C^2$ expanding unimodal map $f_t$ admits a unique
absolutely continuous invariant probability measure $\mu_t=\rho_t\, \D
x$.  The measure $\mu_t$ is ergodic, and it is mixing if $f_t$ is
mixing. The density $\rho_t$ is the unique fixed point (see
e.g.\ \cite{Book}) of the transfer operator $\LL_t$ defined on $BV$ by
\[
\LL_t \varphi (x) = \sum_{f_t (y) = x} \frac{\varphi (y)}{|f_t'
  (y)|} \, .
\]

\smallskip

We now recall the definition of the Marchaud fractional derivatives in
 order to introduce fractional susceptibility functions. 
Let $ 0 < \eta < 1$, and set 
$\Gamma_\eta=\frac{\eta}{\Gamma (1-\eta)}$
where $\Gamma$ is  Euler's function.  Let $g \colon \bR \to \complex$ be a bounded globally $\bar\eta$-H\"older function  with $\bar\eta \in (0,1)$.
Recall 
\cite[p.~110, Theorem~5.9]{sam} that for any $\eta \in (0, \bar \eta)$,
the left-sided and right-sided Marchaud
derivatives of $g$ at $t_0\in \bR$ are 
\[
M_+^\eta g(t_0) = \Gamma_\eta\int_0^\infty
\frac{g(t_0) - g(t_0-t)}{t^{1+\eta}} \, \D t
\, , \, \, \,
M_-^\eta g(t_0) = \Gamma_\eta \int_0^\infty
\frac{g(t_0) - g(t_0+t)}{t^{1+\eta}} \, \D t \, .
\]
The two-sided Marchaud derivative is then defined by
\begin{align*}
  M^\eta g (t_0) &:= \frac{1}{2} \bigl (M_+^\eta g(t_0) - M_-^\eta g(t_0)
\bigr )
  \\
  &=\frac{\Gamma_\eta}{2} \biggl(
  \int_0^\infty \frac{g(t_0) - g(t_0-t)}{t^{1+\eta}} \, \D t -
  \int_0^\infty \frac{g (t_0) - g(t_0+t)}{t^{1+\eta}} \, \D t  \biggr) \\
   &=
  \frac{\Gamma_\eta}{2}
  \int_{-\infty}^\infty \frac{g (t_0+t) - g(t_0)}{|t|^{1+\eta}} \sign(t)
    \, \D t \, .
\end{align*}

The choice
of the normalisation $\Gamma_\eta$ ensures
the following key property:

\begin{lemma}\label{limit}
Assume that $g$ is bounded on $\mathbb{R}$ and differentiable at $t_0$. Then
  \[
  \lim_{\eta \uparrow 1} M_+^\eta g (t_0) = g'(t_0) \,\, \mbox{ and } \,\,
   \lim_{\eta \uparrow 1} M_-^\eta g (t_0) = - g'(t_0) \, , \,
\mbox{ so that } \,\, \lim_{\eta \uparrow 1} M^\eta(g)(t_0)=g'(t_0)\, .
  \] 
\end{lemma}

(The above lemma is certainly well-known, we provide
a proof  in Appendix~\ref{ll}.)

\medskip

Given a $C^2$ perturbation $(f_t)_{|t| \le \ve}$ of a piecewise $C^2$
expanding unimodal map $f$, and given $\ve_1\le \ve$, we put\footnote{See also \eqref{truncc} for an
alternative approach.}
\begin{equation}
\label{???}
f_t^{(\ve_1)} = f_t \, , \forall
|t| \le \ve_1\, , \quad 
f_t^{(\ve_1)} = f_{\ve_1}\, , \forall t \ge \ve_1\, , \quad
f_t^{(\ve_1)} =
f_{-\ve_1}\, , \forall  t \le -\ve_1\, .
\end{equation}
  Then, for any function $\phi\in L^1(I)$, we define
the \emph{$\eta$-fractional susceptibility function} of
$(f_t^{(\ve_1)})$ for the observable $\phi$ to be the formal power
series
\begin{equation}\label{deffr0}
\Psi_\phi (\eta, z) =
 \sum_{k=0}^\infty \frac{\Gamma_\eta}{2} z^k \int_I \phi
\int_{-\infty}^\infty
\LL_t^k  \biggl( \frac{(\LL_{ t}-\LL_0 )
  \rho_0}{|t|^{1+\eta}} \biggr) \sign(t)  \, \D t \,
\D x\, ,
\end{equation}
and we define the \emph{frozen
$\eta$-fractional susceptibility function} of $(f_t^{(\ve_1)})$ and $\phi$ to be the formal  series 
\begin{equation}\label{deffr}
  \Psi^\mathrm{fr}_\phi(\eta,z)= 
  \sum_{k=0}^\infty  \frac{ \Gamma_\eta}{2} z^k\int_I \phi
     \int_{-\infty}^\infty 
     \LL_0^k  \biggl( \frac{(\LL_{ t}-\LL_0 )
       \rho_0}{|t|^{1+\eta}} \biggr)  \sign(t) \, \D t \,
     \D x \, .
\end{equation}
The coefficient of 
$z^k$ in each of the two formal power series above is a sum of improper integrals,
for $t\in (-\infty,0)$ and $t\in (0, \infty)$. We shall see in the proof of Theorem~\ref{thm:susc_holom} that each integral
converges, if  $\phi$ belongs to $L^q$ for large enough $q$. In particular, using Fubini,
we recover the formulas
\eqref{1} and \eqref{2} stated in the introduction.

\medskip

We are now ready to state our  main result. Recall $\Lambda>1$ from
Definition~\ref{defL}. For a function $g(x,t)$ of two real variables, we denote
by $M^{\eta} g   \big|_{t=t_0}$  the Marchaud derivative of $g$ in the $t$ variable,
at $t=t_0$.

\begin{theorem}[Fractional susceptibility function]\label{thm:susc_holom}
Let $(f_t)_{|t| \le \ve}$ be a $C^{2}$ perturbation  of a
 mixing piecewise $C^2$ expanding
unimodal map $f$.  
Assume that either $f$ is good or that all the $f_t$ are topologically
conjugated to $f$. Then there exist $\kappa <1$ (depending only
on $f_0$) and $\ve_1\in (0, \ve)$ 
 such that for any
 $0<\eta <1$, 
and for any $\phi \in L^q(I)$ with  $q >(1-\eta)^{-1}$, the following holds for 
the fractional susceptibility functions of $(f_t^{(\ve_1)})$ and $\phi$:

\begin{enumerate}[\itshape(a)]
\item\label{main} 
  $\Psi_\phi (\eta, z)$ 
  and  $\Psi^\mathrm{fr}_\phi (\eta, z)$  are  holomorphic in the open disc of
  radius $ \min( \Lambda^{1-\eta}, \kappa^{-1})$. 

\smallskip

\item\label{eq:response}
    $  \Psi_\phi (\eta, 1)=   M^{\eta} \bigl( \int_I \phi(x) \rho_t(x) \, \D x \bigr)    \big|_{t=0}$.

\smallskip 
\item\label{eq:last} $\Psi^\mathrm{fr}_\phi (\eta, z)=
\int_I \phi(x) \cdot \bigl ( ( I - z \LL_0)^{-1} M^\eta (\LL_t \rho_0)|_{t=0} \bigr ) (x)\,  \D x$ for all $|z|\le 1$.
\end{enumerate}
\end{theorem}

As explained in \S\ref{1.2}, formula (\ref{eq:response}) in the above theorem
can be viewed as the key result of this work.
The proof of Theorem~\ref{thm:susc_holom} will be given in Section~\ref{pr},
after we introduce some necessary tools in Section~\ref{tools}.
We next make a few remarks about the statement:

Clearly,  if $\phi \in L^\infty$, then  we can consider all
values of $\eta \in (0,1)$.

The proof of Theorem~\ref{thm:susc_holom} actually shows that for
$\phi \in L^q (I)$, the response function is H\"{o}lder
continuous with exponent $\eta$ provided that $q > (1 - \eta)^{-1}$.

We do not claim that the holomorphy radius given in claim a) is optimal. 
However, we expect that, generically, the maximal holomorphic extension
radius of   $\Psi_\phi (\eta, z)$ tends to one as $\eta \to 1$.
It is unclear whether the frozen fractional susceptibility function
is holomorphic in a disc of radius larger than one, uniformly in $\eta \to 1$.

Claim c) implies that, for any $\eta \in (0,1)$, and all $|z|\le 1$
\begin{align} \label{nicefrozen}
\Psi^\mathrm{fr}_\phi (\eta, z)&= \sum_{k=0}^\infty z^k
\int_I \phi(f_0^k(x)) \cdot  \bigl ( M^\eta (\LL_t \rho_0)|_{t=0} \bigr ) (x) \, \D x \, .
\end{align}
For $z=1$ this is reminiscent of the fluctuation--dissipation formula
for linear response (see e.g.\ \cite{BKL}). Note, however, that we
cannot integrate by parts (in spite of \cite[(6.27)]{sam}) because the
Marchaud derivative is with respect to the parameter $t$. See
Corollary~\ref{lecor} for more information on the frozen
susceptibility function in the ``horizontal'' case.

\medskip
To state an interesting corollary of our
Theorem~\ref{thm:susc_holom}, letting $H_u$ denote the Heaviside
jump at $u\in \real$, i.e., $H_u(x)=-1$ if $x< u$, while
$H_u(x)=0$ if $x >u$, and $H_u(u)=-1/2$, and setting
$c_k=f^k(c)$, we recall that \cite[Proposition~3.3]{Baladi2007}
we may decompose $\rho_0$ as $ \rho_0 = \rho_0^\mathrm{reg} +
\rho_0^\mathrm{sal} $, where the regular part
$\rho_0^\mathrm{reg}$ is differentiable and supported in $I$,
with derivative in $BV$, and the saltus (or singular) part
$\rho_0^\mathrm{sal} = \sum_{k=1}^{N_f} s_k H_{c_k}$ where
$N_f=\# \{\, c_k \mid k \ge 1 \,\}\in [2, \infty]$.  In addition,
if the critical point $c$ is not periodic, we have
\begin{equation}\label{rsal}
\rho_0^\mathrm{sal} = \sum_{k=1}^{\infty} \bar s_k H_{c_k}\, ,
\quad \mbox{ where } \,  \bar s_1=-\lim_{x \uparrow c_1} \rho_0(x)<0\, ,
\, \, \bar s_k=\frac{s_1 }{(f^{k-1})'(c_1)}\, .
\end{equation}
Note that if $N_f$ is finite but
$c$ is  not periodic (it is then preperiodic)
then the jump $s_j$ at $c_j$ is given by the sum of
all $\bar s_k$ for $k$ such  that $c_k=c_j$.

 We say that  a bounded function $v$  is \emph{horizontal} for $f$ if, setting $P_f=N_f$ if $c$ is periodic with minimal period $P_f\ge 2$,
and $P_f=\infty$ otherwise, we have
\begin{equation}\label{horr}
\sum_{k=0}^{P_f-1} \frac{ v (c_k)}{(f^{k})'(c_1)}
=0 \, .
\end{equation}
If $v$ is not horizontal, we say that $v$ is \emph{transversal.}
In the horizontal case, we have (Corollary~\ref{lecor} is proved in Section~\ref{proofcor}):

\begin{corollary}[Fractional susceptibility of horizontal perturbations]
  \label{lecor}
  Let $(f_t)_{|t| \le \ve}$ be a $C^{3}$ perturbation of a mixing
  piecewise $C^3$ expanding unimodal map $f$.  Assume that either
  $f$ is good or that all the $f_t$ are topologically conjugated
  to $f$.  Assume in addition\footnote{This implies
    $v_0(-1)=v_0(1)=0$ and if $f(c)=1$ then $v_0(c)=0$.} that
  $f_t(-1)=f_t(1)=-1$ for all $t$, that $ v_0=\partial_t
  f_t|_{t=0} = X_0 \circ f_0 $ is horizontal for $f_0$, where
  $X_0$ is $C^2$ on $f(I)$, and, finally that $(x,t)\mapsto
  f_t(x)$ extends as a $C^{2}$ map to $(\tilde I_+ \cup \tilde
  I_-) \times [-\ve, \ve]$.  Then, there exists $\ve_2\in (0,
  \ve)$ such that for any $\phi \in C^0$, the following holds for
  the fractional susceptibility function of $(f_t^{(\ve_2)})$ and
  $\phi$:
  \begin{align}
    \nonumber \lim_{\eta\uparrow 1} \Psi_\phi (\eta, 1)
    &=\partial_t \biggl ( \int \phi(x) \rho_t(x) \, \D x\, \biggr
    ) \biggr |_{t=0} \\
    \label{reTCE} &=-\sum_{j=1}^{P_f} \phi(c_j)  \sum_{k=1}^j   \bar s_k X_0(c_k)
    - \int \phi \cdot (I- \LL_0)^{-1} (X_0'
    \rho_0^\mathrm{sal}+(X_0 \rho_0^\mathrm{reg})') \, \D x \, .
  \end{align}
  Assume furthermore that $c$ is not periodic for $f$.  Then, we
  have
  \begin{equation}\label{noper}
    \lim_{\eta \uparrow 1}\bigl ( (\Psi_\phi (\eta, 1) -
    \Psi^\mathrm{fr}_\phi (\eta, 1) \bigr )=0 \, .
  \end{equation}
\end{corollary}

\begin{remark}
  The first term of \eqref{reTCE} can be rewritten as
  \[
  -\sum_{j=1}^{P_f} \phi(c_j) \sum_{k=1}^j \bar s_k X_0(c_k)=
  -\int \phi \alpha (\rho_0^\mathrm{sal})'\, , \,\, \mbox{ where }
  \alpha(x):= -\sum_{j=0}^{P_f-1} \frac{v(f^{j}(x))}{(f^{j+1})'
    (x)} \, .
  \]
  Indeed, $\alpha$ is the solution (which \cite[Lemma~2.2,
    Remark~2.3, Proposition~2.4]{BS} is unique and continuous
  under the assumptions of the corollary) of the twisted
  cohomological equation
  \begin{equation}\label{TCE}
    X_0(f(x))=v_0(x)=\alpha(f(x))-f' (x) \alpha(x) \, , \, x \ne
    c \, , \, \, \, X_0(c_1)=v(c)= \alpha(c_1) \, ,
  \end{equation} 
  so that, for $j \le P_f$,
  \begin{equation}\label{magic}
    \sum_{k=1}^{j} \bar s_k X_0(c_k) =\bar s_1 \sum_{k=1}^{j}
    \frac{ X_0(c_k)}{(f^{k-1})' (c_1)} = \bar s_1\biggl (
    X_0(c_1) - \alpha(c_1) + \frac{\alpha(c_j)}{(f^{j-1})' (c_1)}
    \biggr )= \bar s_j \alpha(c_j)\, .
  \end{equation}
\end{remark}

\subsection{Fractional moduli of continuity and Keller's $x \log x$ bound}
\label{KellerlogS}

In the definition of the Marchaud derivatives $M_+^\eta g$ and
$M_-^\eta g$ of a function $g$, we may replace $t^{1+\eta}$ with other
weights. Suppose for instance that $\ell \colon [0,\infty) \to
  \mathbb{C}$ is a continuous function and that $\gamma \in [0,1)$ is a
  constant such that
\begin{equation}\label{hh}
  \ell(0)=0, \qquad \int_0^1 \frac{t^\gamma}{|\ell (t)|} \, \D
  t < \infty \qquad \mathrm{and} \qquad \int_1^\infty
  \frac{1}{|\ell(t)|} \, \D t < \infty \, .
\end{equation}
We then define the right and left-sided $\ell$-Marchaud derivatives of
a bounded $\gamma$-H\"older function $g$ by
\[
M_+^{(\ell)} g (t_0) = \int_0^\infty \frac{g(t_0) - g(t_0-t)}{\ell(t)}
\, \D t\, , \qquad M_-^{(\ell)} g (t_0) = \int_0^\infty \frac{g(t_0) -
  g(t_0+t)}{\ell(t)} \, \D t \, .
\]
We define the two-sided $\ell$-Marchaud derivative by $M^{(\ell)} g =
\frac{1}{2} (M_+^{(\ell)} g - M_-^{(\ell)} g)$.
Finally, we define the $\ell$-susceptibility function to be the formal
power series
\[
\Psi_\phi ((\ell),z) = \sum_{k=0}^\infty z^k \int_I \phi
\int_{-\infty}^\infty\biggl( \LL_t^k \frac{(\LL_t - \LL_0)
  \rho_0}{\ell(|t|)} \biggr) \sign(t)  \, \D t\, \D x\, .
\]
Similarly, we define the frozen $\ell$-susceptibility function by the
formal power series
\[
\Psi_\phi^\mathrm{fr} ((\ell),z) = \sum_{k=0}^\infty z^k \int_I \phi
\int_{-\infty}^\infty\biggl( \LL_0^k \frac{(\LL_t - \LL_0)
  \rho_0}{\ell(|t|)} \biggr) \sign(t)  \, \D t\, \D x\, .
\]

The following theorem is proved in almost the same way as
Theorem~\ref{thm:susc_holom} (see Section~\ref{sec:generalizedproof}).

\begin{theorem}[Generalized fractional susceptibility function]\label{newth}
  Let $(f_t)_{|t| \le \ve}$ be a $C^{2}$ perturbation of a mixing
  piecewise $C^2$ expanding unimodal map $f$. Assume that either $f$
  is good or all the $f_t$ are topologically conjugated to $f$.  Let
  $\ell$ and $\gamma \geq 0$ be such that \eqref{hh} holds.  Then, for
  any $q\ge (1-\gamma)^{-1}$ and any $\phi \in L^q$, the following holds:
  \begin{enumerate}[\itshape(a)]
  \item \label{maingeneralized} the susceptibility functions
    $\Psi_\phi ((\ell), z)$ and $\Psi_\phi^\mathrm{fr} ((\ell), z)$
    are well-defined and holomorphic in a disc of radius strictly
    larger than one.
  \item \label{responsegeneralized}  $\Psi_\phi ((\ell), 1) = M^{(\ell)} \bigl( \int_I \phi(x)
    \rho_t(x) \, \D x \bigr) \big|_{t=0}$.
  \item \label{lastgeneralized} $\Psi^\mathrm{fr}_\phi ((\ell), z)=
    \int_I \phi(x) \cdot \bigl ( ( I - z \LL_0)^{-1} M^{(\ell)} (\LL_t
    \rho_0)|_{t=0} \bigr ) (x) \, \D x$ for all $|z|\le 1$.
  \end{enumerate}
\end{theorem}

For instance, (cf.\ the power-logarithmic kernel operators discussed
in \cite[\S 21]{sam}) we may fix $0 < \eta < 1$ and $\beta \geq 0$ and
set
\[
\ell(t)=\ell_{\eta, \beta} (t) = t^{1+\eta} | \min \{-1, \log t
\}|^\beta\, .
\]
Then, if $\phi \in L^q$, $\eta = 1 - \frac{1}{q}$ and $\beta > 1$,
Theorem~\ref{newth} for $\gamma < \eta$ implies
\[
M^{(\ell_{\eta, \beta})} \biggl( \int_I \phi(x) \rho_t(x) \, \D x
\biggr) \bigg|_{t=0} = \Psi_\phi ((\ell_{\eta,\beta}), 1)\, .
\]
In particular, if $q=\infty$, we may take $\eta=1$ and $\beta >1$
arbitrarily close to $1$ which is reminiscent of Keller's \cite{Ke}
$t|\log |t||$ modulus of continuity.

Another consequence of Theorem~\ref{newth} is that we may consider the
Marchaud derivative $M^\eta \bigl( \int_I \phi (x) \rho_t (x) \,
\D t \bigr) \big|_{t=0}$ and the susceptibility
functions $\Psi_\phi (\eta, z)$ and   $\Psi_\phi^\mathrm{fr}(\eta, z)$  for non-real values of
$\eta$.  A corollary
of Theorem~\ref{newth} is the following result, which we prove in
Section~\ref{sec:realanalyticproof}.

\begin{corollary} [Holomorphic extension in $(\eta,z)$]\label{cor:realanalytic}
 Let $(f_t)_{|t| \le \ve}$ be a $C^{2}$ perturbation of a mixing
  piecewise $C^2$ expanding unimodal map $f$. Assume that either $f$
  is good or all the $f_t$ are topologically conjugated to $f$, and let $\kappa<1$
  and $\ve_1$ be from Theorem~\ref{thm:susc_holom}.  
  Let $\phi \in L^q$, for some $q > 1$. Then the function
  $\eta \mapsto M^\eta \bigl( \int_I \phi (x) \rho_t (x) \, \D
  t \bigr) \big|_{t=0}$ is holomorphic in the strip $0 < \Re
  \eta < 1 - \frac{1}{q}$, and the functions $(\eta, \zeta) \mapsto
  \Psi_\phi (\eta, z)$ and $(\eta, \zeta) \mapsto
  \Psi_\phi^\mathrm{fr} (\eta, z)$ are  holomorphic in the domain $\{\, (\eta, z) \mid
  0 < \Re \eta < 1-\frac 1 q ,\ |z| < \min (\kappa,
  \Lambda^{-\frac{1}{q}}) \,\}$.
\end{corollary}

\subsection{Two linear examples}
\label{lin}

We  illustrate our definitions
with two families of linear tent maps. 
One simplifying feature is that $\rho^\mathrm{reg}_t$ vanishes identically for all $t$ in both examples.
To study $\rho^\mathrm{sal}_0$, we shall use that for any $x\ne u$, we have
\begin{equation} \label{oldstuff}
M^\eta_t H_{t}(x)|_{t=u} = - \frac{1}{2 \Gamma (1 - \eta)}
\frac{1}{|x-u|^\eta}.
\end{equation}
To prove \eqref{oldstuff}, observe that if $x - u > 0$, then
\[
H_{t+u} (x) - H_u (x) = H_{t} (x-u) - H_0 (x-u) =
\left\{ \begin{array}{rl} 0 & \text{if } t < x-u \\ -1 & \text{if } t
  > x - u \, . \end{array} \right 
.\]
Similarly, if $x - u < 0$, then
\[
H_{t+u} (x) - H_u (x) = H_{t} (x-u) - H_0 (x-u) =
\left\{ \begin{array}{rl} 1 & \text{if } t < x-u \\ 0 & \text{if } t
  > x - u \, . \end{array}  \right .
\]
It follows that
\begin{align*}
M^\eta_t H_{t}(x)|_{t=u} &= \frac{\Gamma_\eta}{2}
\int_{-\infty}^\infty \frac{H_{t+u} (x) - H_u (x)}{|t|^{1+\eta}} \sign
(t) \, \D t \\ &= - \frac{\Gamma_\eta}{2} \int_{|x-u|}^\infty
\frac{1}{|t|^{1+\eta}} \, \D t = - \frac{1}{2 \Gamma (1 -
  \eta)} \frac{1}{|x-u|^\eta}\, ,
\end{align*}
as claimed.
Note also  that since $H_{u+t}(x)= H_u(x-t)$ we have 
\begin{align}\label{old2}
M^\eta_y H_{u}(y)|_{y=x}=-M^\eta_t H_{t}(x)|_{t=u}\, , \quad \forall x\ne u \, .
\end{align}

\smallskip
The first example is the family of tent maps with fixed slopes
$\lambda_0\in (1,2)$ given by $\tilde f_t(x)=\lambda_0 x +
\lambda_0-1+ t_0 + t$ for $-1\le x \le 0$, and $\tilde
f_t(x)=-\lambda_0 x + \lambda_0-1+ t_0 + t$ for $0\le x \le 1$, where
$|t|<\ve_1$, for small enough $\ve_1\in (0, \min \{|t_0|, 2-\lambda_0
- t_0\})$. (For $|t| > \ve_1$ we set $\tilde f_t(x)=\tilde
f_{\pm\ve_1}(x)$.)  This family is horizontal because each $\tilde
f_t$ is topologically conjugated to $\tilde f_0$. (Indeed, setting,
$h_t (x) = (\lambda_0 - 1 + t_0 + t) x$, we have $h_t \circ
\tilde{f}_0 (x) = \tilde{f}_t \circ h_t (x)$.  We can also use that
the topological entropy of $\tilde f_t$ is $\log \lambda_0$ for all
$t$.)

We assume that the orbit of $c=0$ is infinite for the sake of simplicity.
We have $\tilde X_0\equiv 1$ on the support of $\tilde \rho_0$, so that 
\eqref{oldstuff} and \eqref{old2} imply 
\[
M^\eta_x (\tilde X_0 \tilde \rho_0) (x) =\sum_{n=1}^{\infty} \bar s_n
M^\eta_x H_{c_n}(x)= \frac{1}{2 \Gamma (1-\eta)} \sum_{n=1}^{\infty}
\frac{s_1}{(\tilde{f}_0^{n-1})'(c_1)} \frac{1}{|x-c_n|^\eta} \, ,
\]
where $(\tilde{f}_0^{n-1})'(c_1)=\sigma_n \lambda_0^{n-1}$, for
$\sigma_n=\sgn (\tilde{f}_0^{n-1})'(c_1)$.  It follows that
\begin{align*}
  \Psi^\mathrm{rsp}_\phi(\eta,z)&=- \frac{1}{2 \Gamma (1-\eta)}
  \sum_{k=0}^\infty z^k \int_I (\phi \circ \tilde f_0^k)
  \sum_{n=1}^{\infty} \frac{s_1 \sigma_n}{\lambda_0^{n-1}}
  \frac{1}{|x-c_n|^\eta} \, \D x
\, .
\end{align*}
Next, observe that since $\tilde f_t(y)=\tilde f_0(y)+t$
for $|t|<\ve_1$, we have $ (\tilde
\LL_t \varphi)(x)= (\tilde \LL_0 \varphi)(x-t)$
for such $t$, so that, using
$\tilde \LL_0 \rho^\mathrm{sal}_0= \rho^\mathrm{sal}_0$, we find, for
any $|t|<\ve_1$,
\begin{equation}\label{magic2}
  (\tilde \LL_t \tilde \rho_0 )(x)=\sum_{n=1}^{\infty}
  \frac{s_1}{(\tilde{f}_0^{n-1})'(c_1)} H_{c_n}(x-t)=\sum_{n=1}^{\infty}
  \frac{s_1 \sigma_n}{\lambda_0^{n-1}} H_{t}(x-c_n)\, .
\end{equation}
Set $H^{(\ve_1)}_{t}(y)=
H_{t}(y)$ if $|t|<\ve_1$, and $H^{(\ve_1)}_{t}(y)=H_{\pm \ve_1}(y)$ if
$\pm t >  \ve_1$.
Then,  we have,
\begin{align}  \label{oldstuff1}
  M^\eta_t H^{(\ve_1)}_{t}(y)|_{t=u} =
  \left\{ \begin{array}{rl} - \frac{1}{2 \Gamma (1 - \eta)}
  \frac{1}{|y - u|^\eta}  \, & \text{if }  0 <|y-u|< \ve_1
\, , 
  \\ 
  0 \qquad\qquad\qquad
  & \text{if }  |y-u|\ge  \ve_1  \, .
   \end{array}  \right .
\end{align}
(The claims above are proved in the same way as \eqref{oldstuff}.) 
Using \eqref{magic2} and \eqref{oldstuff1}, we get, 
\[
M^\eta_t (\tilde \LL_t \tilde \rho_0)(x)|_{t=0}=-\frac{1}{2\Gamma (1-\eta)}
\sum_{n=1}^{\infty} \frac{s_1 \sigma_n}{\lambda_0^{n-1}} \biggl (
\frac{1_{[c_n-\ve_1, c_n+\ve_1]}(x)}{|x-c_{n}|^\eta} \biggr ) \, .
\]
 Finally,
\begin{align*}
\Psi^\mathrm{fr}_\phi(\eta,z)
&=-
\frac{1}{2\Gamma (1-\eta)} \sum_{k=0}^\infty z^k
\int_I (\phi \circ \tilde f_0^k)   \sum_{n=1}^{\infty} \frac{s_1 \sigma_n}{\lambda_0^{n-1}} 
 \frac{1_{[c_n-\ve_1, c_n+\ve_1]}(x)}{|x-c_{n}|^\eta}   \,    \D x
\, .
\end{align*}
So we see that, in the horizontal linear case given by the family
$(\tilde f_t)$, the frozen and response susceptibility function
coincide if $\epsilon_1\ge \sup_n (c_1-c_n, c_n-c_2)$. This does
not seem possible, since $\ve_1 \le \ve_0$, where $\ve_0$ is
given by the uniform Lasota--Yorke bound
Lemma~\ref{lem:lasota_yorke}. However, the two susceptibility
functions are qualitatively similar. (In fact, the relation $\tilde f_t=\tilde f_0 + t$ for $|t|<\epsilon_1$ implies that the response and frozen susceptibilities differ by a function holomorphic in a disc of radius larger than one; this can be shown as in \cite[Proposition~2.5]{BS3}). In addition, if we
replaced\footnote{With this modification, the use of
  $f_t^{(\ve_1)}$ as defined in \eqref{???} would not be needed.}
in the definitions of all fractional susceptibility functions the
Marchaud derivative by the truncated Marchaud derivative
 \begin{equation}\label{truncc} 
M^{\eta, (\ve_1)} g(t_0):=  \frac{\Gamma_\eta}{2}
  \int_{-\ve_1}^{\ve_1} \frac{g (t_0+t) - g(t_0)}{|t|^{1+\eta}} \sign(t)
    \, \D t  \, , 
\end{equation}
then the frozen and response susceptibility functions would coincide for $\tilde f_t$.
(Lemma~\ref{limit} and the integration by parts
formula mentioned in footnote~\ref{samm} both hold for $M^{\eta, (\ve_1)}$.)

\medskip

The second example is the family of tent maps with varying slopes
$\bar f_t(x)=\lambda_t x + \lambda_t-1$ for $-1\le x \le 0$ and $\bar
f_t(x)=-\lambda_t x + \lambda_t-1$ for $0\le x \le 1$, where
$\lambda_t=\lambda_0+t$, for $\lambda_0 \in(1,2)$ and $|t|<\ve_1$,
with small enough $\ve_1<\min (2-\lambda_0, \lambda_0-1)$.  This
family is not horizontal (see e.g.\ \cite[\S4]{Tsu}).  We assume again
that the orbit of $c=0$ for $\bar f_0$ is infinite for the sake of
simplicity.
 
 Setting $\bar v_t=\partial_s\bar f_s|_{s=t}$, and defining $\bar X_t$ by $\bar v_t=\bar X_t\circ \bar f_t$,
we have $\bar v_t(x)=x+1$ if $-1<x<0$ and $\bar v_t(x)=-x+1$ if $0<x<1$,
so that $\bar X_t(y)=(y+1)/\lambda_t$ and $\bar X_0(y)=(y+1)/\lambda_0$, for $y\in [-1,1]$.  
We find, 
 \begin{align}\label{disc1}
 M^\eta_x (\bar X_0 \bar \rho_0) (x)&=
 M^\eta_x\biggl ( \sum_{n=1}^{\infty} \frac{s_1 \sigma_n }{\lambda_0^{n-1}}
  \frac{x+1}{\lambda_0} H_{c_n}(x)\biggr ) \, .
 \end{align}
 (The computation of the Marchaud derivative above is straightforward, but cumbersome,
 using that the sum over $n$ is supported in $[c_2, c_1]$,
 and we do not carry it out here. We just point out that, if we use  $M^{\eta, (\ve_1)}$ from \eqref{truncc}
 instead of $M^\eta$, we have  
\begin{align*} M^{\eta, (\ve_1)}( x H_{c_n}(x))&=
\frac{\Gamma_\eta }{2}\int_{-\ve_1}^{\ve_1} \sign(y)
\frac{ (x+y) \cdot  H_{c_n}(x+y)	-   x \cdot  H_{c_n}(x)}{|y|^{1+\eta}} \, \D y
\\
&=\frac{\Gamma_\eta }{2}\int_{-\ve_1}^{\ve_1} 
\frac{   H_{c_n-x}(y)	}{|y|^{\eta}} \, \D y
+ x \cdot M^{\eta, (\ve_1)}( H_{c_n}(x)) \, ,
\end{align*}
so that the singularity type of $M^{\eta, (\ve_1)}(( x+1 )H_{c_n}(x))$
is $|x-c_n|^{-\eta}$.)

\smallskip
 
We move to the frozen susceptibility function.
We have $\bar f_t(x)=(\lambda_t/\lambda_0) (\bar f_0(x)+1) -1$,
so that
\begin{equation}\label{trr}
\bar \LL_t \varphi(x)= \frac{\lambda_0}{\lambda_t} \cdot 
\bar \LL_0 \varphi\biggl ( \frac{\lambda_0 x -t }{\lambda_t} \biggr  ) \, ,
\end{equation}
for all $|t|<\ve_1$.
Combining \eqref{trr} with $\bar \LL_0 \bar \rho_0^\mathrm{sal}=\bar \rho_0^\mathrm{sal}$, we find
\begin{align*}
\bar \LL_t \bar \rho_0^\mathrm{sal}(x)&=\frac{\lambda_0}{\lambda_t}
\bar \rho_0^\mathrm{sal} \biggl ( \frac{\lambda_0 x -t }{\lambda_t} \biggr  )
= \frac{\lambda_0}{\lambda_0+t} \sum_{n=1}^{\infty} \frac{s_1 \sigma_n }{\lambda_0^{n-1}} 
   H_{c_n}  \biggl ( \frac{\lambda_0 x -t }{\lambda_0+t} \biggr  ) \\
 &= \frac{1}{1+t/\lambda_0} \sum_{n=1}^{\infty} \frac{s_1 \sigma_n}{ \lambda_0^{n-1}} 
 H_{c_n+(c_n+1)t/\lambda_0}   ( x ) \\
  &= \biggl (\sum_{k=0}^\infty \bigl (- \frac{t}{\lambda_0}\bigr)^k\biggr )
  \cdot  \sum_{n=1}^{\infty} \frac{s_1 \sigma_n}{ \lambda_0^{n-1}} 
 H_{t}  \biggl (\lambda_0 \frac{x-c_n}{c_n+1} \biggr )  \, ,
\end{align*}
for all $|t|<\ve_1$, so that
\begin{equation}
  \label{disc2}
  M^\eta_t (\bar \LL_t (\bar \rho_0)(x))|_{t=0}=
  \sum_{k=0}^\infty \frac{(-1)^k}{\lambda_0^k} M^\eta_t \biggl (
  t^k \cdot \sum_{n=1}^{\infty} \frac{s_1 \sigma_n}{
    \lambda_0^{n-1}} H_{t} \biggl (\lambda_0 \frac{x-c_n}{c_n+1}
  \biggr ) \biggr)\bigg |_{t=0} \, .
\end{equation}
Comparing \eqref{disc1} with \eqref{disc2}, we see that the
discrepancy between $M^\eta_x (\bar X_0 \bar \rho_0) (x)$ and
$-M^\eta_t (\bar \LL_t (\bar \rho_0)(x))|_{t=0}$, and thus
between $\Psi^\mathrm{fr}_\phi(\eta, z)$ and
$\Psi^\mathrm{rsp}_\phi(\eta, z)$, is more marked for the family
$\bar f_t$ than for the family $\tilde f_t$.


\section{Transfer operators and Sobolev spaces}
\label{tools}

\subsection{Basic definition and properties}

\label{standard}

We will use the  Sobolev spaces $\HH^{\tau,p}=\HH^{\tau,p}(I)$ of functions $\varphi
\in L^p(\real)$ supported in $I$ such that
\[
\| \varphi \|_{\HH^{\tau,p}} = \| \mathscr{F}^{-1} ((1 +
|\xi|^2)^{\tau/2}( \mathscr{F} \varphi)) \|_{L^p} < \infty \, ,
\]
where $p>1$ and $\tau\in [0, 1/p)$ are real numbers, and $\mathscr{F}$ denotes the Fourier transform. 
(Note that $\HH^{\tau,p}(I)\subset \HH^{\tilde \tau,p}(I)$
if $\tilde \tau < \tau$, and that the embedding is compact.)
 
Let $f_t$ be a piecewise $C^2$ expanding unimodal map.
Thomine \cite{Thomine} showed that the transfer operator
$\LL_t$ is bounded on $ \HH^{\tau,p}$, for any 
$1<p <\infty$ and $0 \le  \tau
< \frac{1}{p}$. 
More precisely, there exists a constant $C=C(\sup |f''_t|, \inf |f'_t|)$ such that
 \begin{equation}
\label{unifop}\| \LL_t \varphi \|_{\HH^{\tau,p}} \leq C 
    \| \varphi \|_{\HH^{\tau,p}}\, ,
\quad  \forall \varphi \, , \, 
  \forall 0\le \tau <1/p \, .
\end{equation}

Recall that the
essential spectral radius
$r_\mathrm{ess}(\LL|_\cB)$  of a bounded operator  $\cL \colon \cB \to \cB$  on a Banach space
$\cB$ is  the smallest $r\ge 0$ such that
the spectrum of $\cL$ on $\cB$ in the complement  of the disc of radius $r$
 consists of isolated
    eigenvalues of finite multiplicity.
Thomine \cite[Theorem~1.3]{Thomine} proved that,
 for any $1<p <\infty$ and $0 < \tau
< \frac{1}{p}$, we have\footnote{Indeed, in our one-dimensional unimodal setting Thomine's ``$n$-complexity at the beginning'' is bounded by $2$, and his
``$n$-complexity at the end'' is bounded by $2^n$.
Cf.\ the proof of our uniform Lasota--Yorke estimate Lemma~\ref{lem:lasota_yorke}.}
\begin{align*}
  r_\mathrm{ess}(\LL_t|_{\HH^{\tau,p}(I)}) 
  &\leq \lim_{n \to  \infty} 
   \Bigl\| 2^{\frac 1 p}\cdot 2^{n(1- \frac1p)} |(f_t^n)'|^{-1 + \frac{1}{p}} (\lambda_n(f_t))^{-\tau } \Bigr\|_\infty^\frac{1}{n} \\
    &\leq \lim_{n \to \infty}
2^{1- \frac1p} (\lambda_n(f_t))^{(-\tau - 1 + \frac1p)/n}\, .
\end{align*}

In particular, if $\tilde \Lambda \in (1, \Lambda(f_t))$,
we find $p(f_t)=p(f_t,\tilde \Lambda)>1$ such that
\begin{align}\label{thom} r_\mathrm{ess}(\LL_t|_{\HH^{\tau,p}}) < 
 \tilde \Lambda^{- \tau - 1 +   \frac1p} <1  \, ,
 \quad \forall\,  1<p <p(f)\, ,\,\forall \,  0 < \tau
 < \frac{1}{p}\, .
\end{align}
By standard arguments (see \cite[Theorem~1.6]{Thomine}), using that $f_t$ is unimodal 
and thus topologically transitive on $[f^2_t(c), f_t(c)]$, and that the dual of $\LL_t$ preserves Lebesgue measure $\D x$
on $I$, it follows that for such $\tau$ and  $p$ the spectral radius
of $\LL_t$ on $ {\HH^{\tau,p}}$ is 
equal to one, that the invariant density
 $\rho_t$ belongs to $\HH^{\tau,p}$, 
(extending $\rho_t$ by zero outside of its domain),  
that  $\rho_t$ is the unique  fixed point of $\LL_t$
in ${\HH^{\tau,p}}$, and that the algebraic multiplicity
of the eigenvalue $1$ is equal to one. 

If $0<\tilde \tau < \tau$ then, outside of the closed disc of radius $\tilde \Lambda^{- \tilde \tau - 1 +   \frac1p}$,
 the spectrum of $\LL_t$ on 
 ${\HH^{\tilde \tau,p}}(I)$ and $ {\HH^{\tau,p}}(I)$ coincide, including
 multiplicity and generalised eigenspaces (this follows from the
 fact that both spaces are continuously embedded in $L^p$
 while $C^1$ functions are dense  in both spaces \cite[2.1.3, Proposition~1]{RS}, applying the result
 in \cite[App.~A]{BaladiTsujii}, see also \cite[App.~A.2]{BaladiZeta}).

 \subsection{Stability of mixing and mixing rates for good families}
 \label{stabm}

To prove stability of the spectrum  of the perturbation $(f_t)$,
the following uniform Lasota--Yorke lemma is crucial.  (A version of
Lemma~\ref{lem:lasota_yorke} is established for the BV and $L^1$ norm,
see\footnote{Above equation (26) in \cite{BS} there is a mistaken
  reference to Remark 5 in \cite{KellerLiverani} instead.}
e.g.\ \cite[comment above Remark 2]{KellerLiverani}. The proof for our
Sobolev spaces is given in Appendix~\ref{AA}.)

\begin{lemma}[Uniform Lasota--Yorke bound]\label{lem:lasota_yorke}
Let $(f_t)_{|t|\le \ve}$ be a $C^{2}$  perturbation of
  a  piecewise $C^2$ expanding unimodal map $f=f_0$. Assume that either $f$ is good or
all the $f_t$ are topologically conjugated to $f$, and recall $\Lambda>1$ from
Definition~\ref{defL}.
Then for any $\tilde \Lambda < \Lambda$ there exist $p_0\in (1, p(f))$ and $\ve_0\le \ve$ such that for any 
    $p \in (1, p_0)$ and  $0 \le \tilde\tau < \tau < \frac1p$
  there exist finite constants $C_0$ and $C$ 
   such that
\begin{align}\label{eq:lasota_yorke}
  \| \LL_t^n \varphi \|_{\HH^{\tau,p}} \leq C 
   \tilde \Lambda^{(- \tau - 1 +   \frac1p)n}
  \| \varphi \|_{\HH^{\tau,p}} + C C_0^n \| \varphi
  \|_{\HH^{\tilde\tau,p}} \, , \quad \forall\,  |t|<\ve_0\, , \, \, 
  \forall \, n \ge 0\, , \, \, \forall \, \varphi\, .
\end{align}
\end{lemma}

\smallskip
  
  We shall also use the following perturbation estimate, proved 
in \S\ref{theend}:

 \begin{lemma}[Perturbation bound] \label{lem:Lt-L0} Let $(f_t)_{|t|\le \ve}$ be a $C^{2}$ perturbation of
  a piecewise $C^2$ expanding unimodal map $f_0$.
  For any $p>1$ and $0  < \tilde \tau < \frac{1}{p}$ there exists $C<\infty$ such that
  \begin{equation} \label{the:Lt-L0}
    \| (\LL_t - \LL_0) \varphi 
    \|_{\HH^{\tilde{\tau},
        p}} \leq  C|t|^{\tau - \tilde{\tau}} \| \varphi
    \|_{\HH^{\tau, p}} \, , \quad \forall\,  |t|<\ve 
    \, , \, \forall \,  
     \tau \in (\tilde \tau, \frac 1 p )\, , \,\,  \forall\,  \varphi\, .
  \end{equation}
  \end{lemma}

 Assume now that $f_t=f$ is mixing.  Then for any $\HH^{\tau,p}$ with
 $p\in (p(f),1)$ and $\tau< 1/p$, we have that $1$ is the only
 eigenvalue of the transfer operator $\LL_0$ of $f$ on the unit circle
 (adapting e.g.\ the proof \cite[Theorem~3.5]{Book} for $\LL_t$ acting
 on $BV$).  In other words, the operator $\LL_0$ has a spectral gap on
 $ {\HH^{\tau,p}}$. The following notation will be useful:
    
\begin{defn}[Maximal  eigenvalue $\kappa<1$ of  a  mixing piecewise $C^2$ expanding
unimodal map $f$]
\label{defkappa}
Let $f$ be a  mixing piecewise $C^2$ expanding
unimodal map.
If there exist $p>1$ and $\tau< 1/ p$ such that $\LL_0$ on 
$ {\HH^{\tau,p}}$ has an
eigenvalue $\zeta\ne 1$ with $|\zeta| >\Lambda(f)^{-   \tau - 1 +   \frac{1}{ p}}$, then we set
$\kappa(f) <1$ to be the maximal such modulus $|\zeta|$. Otherwise we set $\kappa(f)=0$.
\end{defn}

\smallskip
Recalling the notation from \eqref{thom}, fix  $ \tau< 1/ p$ for $p < p(f,\tilde \Lambda)$. Then  
      for any  $\kappa_0 > \max (\kappa(f), \tilde \Lambda^{-  \tau - 1 +   \frac1p})$
 there exists
$C$       such that 
\begin{equation} \label{eq:decayestimate}
 \| \LL_0^n \varphi \|_{\HH^{\tau ,p}} \leq C
        \kappa_0^n \| \varphi \|_{\HH^{\tau,p}}\, ,
        \quad \forall n \ge 0\, , \, \, \forall  \varphi \in \HH^{\tau,p}_0:= \{\, \varphi \in \HH^{\tau,p}(I) \mid
         \int \varphi \,\D x =0 \,\}\, .
\end{equation}

  Lemmas~\ref{lem:lasota_yorke} and \ref{lem:Lt-L0} put us in a 
 position to apply the results of Keller--Liverani and show stability
of \eqref{eq:decayestimate}:
   Since $\HH^{\tau,p}_0$ is invariant under each $\LL_t$, \cite[Corollary~2~(2)]{KellerLiverani} gives
for any  $\kappa_1 > \max (\kappa (f) , \Lambda^{-  \tau - 1 +   \frac1p})$  
constants\footnote{$C(\tau,p)$  tends to infinity as $\eta$ tends to $1$.}
 $C=C(\tau,p) < \infty$, $\ve_1\in (0,\ve_0)$
such that  
\begin{equation}\label{eq:decayestimate_t}
\|  \LL_t^n \varphi \|_{\HH^{ \tau,p}} \le C \kappa_1^n
\|\varphi \|_{\HH^{ \tau,p}}\, 
   , \quad \forall n \ge 1\, , \, \, \forall \, |t| \le \ve_1 \, , \, \, \forall  \varphi \in \HH^{\tau,p}_0\, .
\end{equation}  
The uniform mixing rate given by the above bound will be crucial to establish
our main theorem.

\subsection{Four basic lemmas and the proof of the perturbation Lemma~\ref{lem:Lt-L0}}
\label{theend}

We end this section by showing the perturbation
Lemma~\ref{lem:Lt-L0}. The proof will use the following four standard
lemmas (the first three lemmas are also instrumental in the proof of
Lemma~\ref{lem:lasota_yorke}):
\begin{lemma}[{Triebel \cite[Section~4.2.2]{Triebel}}] \label{lem:multiplicationbysmooth}
  Suppose that $g \in C^\gamma$ where $\gamma > \tau$, and let $p >
  1$. Then there exists $C = C(\tau, p, \gamma) > 0$ such that
  $
  \| g \varphi \|_{\HH^{\tau,p}} \leq C \| g \|_{C^\gamma}
  \| \varphi \|_{\HH^{\tau,p}}
 $.
\end{lemma}

\begin{lemma}[See e.g.\  {\cite[Lemma~3.3]{Thomine}}] \label{lem:compositionbysmooth}
  Let $T$ be a $C^1$ diffeomorphism  of $\bR$ such that
  $\frac{A}{2} \le |T'(x)| \le 2A$ for some $A>0$ and all $x \in \bR$. Then, for all
  $0 \leq \tau \leq 1$ and $p > 1$ there exists $C = C(\tau,p) > 0$
  such that
  $
  \| \varphi \circ T \|_{\HH^{\tau,p}} \le C A^{\tau - \frac1p}
  \| \varphi \|_{\HH^{\tau,p}} + C A^{-\frac1p} \| \varphi
  \|_{L^p}
  $.
\end{lemma}

\begin{lemma}[{Strichartz \cite[Corollary~I.4.2]{Strichartz}}] \label{lem:Strichartz}
  Suppose that $0 \le  \tau < \frac{1}{p}<1$. Then there is a constant $C =
  C(\tau,p)$ such that 
  $
  \| 1_J \varphi \|_{\HH^{\tau,p}} \leq C \| \varphi
  \|_{\HH^{\tau,p}}
  $ for any interval $J$.
\end{lemma}

\begin{lemma}[See e.g.\  {\cite[Lemma~2.39]{BaladiZeta}}] \label{lem:differenceandcomposition}
  For any $p>1$ and any $C^1$
  map $T
  \colon \bR \to \bR$ there exists $C=C(p, \|T\|_{C^1})$
  such that  
  $
  \| \varphi - \varphi \circ T \|_{\HH^{\tilde{\tau},p}} \leq C \|
  I - T \|_{C^1}^{\tau - \tilde{\tau}} \| \varphi
  \|_{\HH^{\tau,p}}$ for all $0<\tilde{\tau} < \tau<1$ 
and all $\varphi$.
\end{lemma}


\begin{proof}[Proof of  Lemma~\ref{lem:Lt-L0}]
Let $I_t =
f_t(I)$, and put $J_t = I_t \setminus (I_0\cap I_t)$ and $K_t = I_0 \setminus(I_0\cap  I_t)$. (It
is possible that $J_t$ or $K_t$ is empty.)  We have
\[
\LL_t \varphi - \LL_0 \varphi = 1_{I_0 \cap I_t}
(\LL_t \varphi - \LL_0 \varphi) + 1_{J_t} \LL_t \varphi
- 1_{K_t} \LL_0 \varphi,
\]
since $\LL_t \varphi$ is supported in $I_t = J_t \cup (I_0 \cap
I_t)$ and $\LL_0 \varphi$ is supported in $I_0 = K_t \cup (I_0
\cap I_t)$. Hence
\begin{equation} \label{eq:easyandtrickypart}
  \| \LL_t \varphi - \LL_0 \varphi
  \|_{\HH^{\tilde{\tau}, p}} \leq \| 1_{I_0 \cap I_t}
  (\LL_t \varphi - \LL_0 \varphi) \|_{\HH^{\tilde{\tau},
      p}} + \| 1_{J_t} \LL_t \varphi \|_{\HH^{\tilde{\tau},
      p}} + \| 1_{K_t} \LL_0 \varphi \|_{\HH^{\tilde{\tau},
      p}} \, .
\end{equation}

We consider first the term $\| 1_{I_0 \cap I_t} (\LL_t
\varphi - \LL_0 \varphi) \|_{\HH^{\tilde{\tau}, p}}$.
If $x \in I_0 \cap I_t$, both $f_{t,\pm}^{-1} (x)$ and
$f_{0,\pm}^{-1} (x)$ are defined, and we may write
\begin{align*}
 & \| 1_{I_0 \cap I_t} (\LL_t \varphi - \LL_0 \varphi)
  \|_{\HH^{\tilde{\tau}, p}} \\ 
&\quad\leq \| 1_{I_0 \cap I_t} (
   \frac{ \varphi}{|f'_t|} \circ f_{t,-}^{-1} 
-  \frac{\varphi}{|f'_0|} \circ f_{0,-}^{-1} )
  \|_{\HH^{\tilde{\tau}, p}} 
+ \| 1_{I_0 \cap I_t} ( 
  \frac{\varphi}{|f'_t|} \circ f_{t,+}^{-1} 
-  \frac{\varphi}{|f'_0|} \circ f_{0,+}^{-1}
 )
  \|_{\HH^{\tilde{\tau}, p}} \, .
\end{align*}
We 
consider the first term on the right-hand side (the second one is
 treated in the same fashion).
Since $(f_t)$ is a $C^{2}$ perturbation, we may
extend $f_{t,-} \colon [-1,c] \to I_t$ to a
$C^2$-diffeomorphism of $\bR$,
still denoted $f_{t,-} \colon \bR \to \bR$, such that $ \inf_{|t|
  \le \ve} |f_{t,-}'| > 1$, while $f_{t,-}''$ has compact support and is bounded
uniformly in $t$, and $\Vert f_{t,-} - f_{0,-} \Vert_{C^1} = O(|t|)$
 as
$|t| \to 0$. 
Since $0<\tilde \tau <1/p$, by the Strichartz Lemma~\ref{lem:Strichartz},
\[
\| 1_{I_0 \cap I_t}   (\frac{\varphi}{|f'_{t,-}|} \circ f_{t,-}^{-1} - 
 \frac{\varphi}{|f'_{0,-}|} \circ f_{0,-}^{-1}) \|_{\HH^{\tilde\tau, p}} 
\leq C \|  \frac{\varphi}{|f'_{t,-}|} \circ f_{t,-}^{-1} - 
 \frac{\varphi}{|f'_{0,-}|} \circ f_{0,-}^{-1}
\|_{\HH^{\tilde\tau, p}}\, ,
\]
where $\varphi$ is
extended by zero outside of its domain $I$. We then split
\begin{align} \label{eq:easypart}
   \|  \frac{\varphi}{|f'_{t,-}|} \circ f_{t,-}^{-1} - 
 \frac{\varphi}{|f'_{0,-}|} \circ f_{0,-}^{-1}
\|_{\HH^{\tilde\tau, p}}
  &
  \leq
 \|  \frac{\varphi}{|f'_{t,-}|} \circ f_{t,-}^{-1} - 
 \frac{\varphi}{|f'_{t,-}|} \circ f_{0,-}^{-1}
\|_{\HH^{\tilde\tau, p}}\\
\nonumber &\qquad+
 \|  (\frac{1}{|f'_{t,-}|} \circ f_{0,-}^{-1} - 
 \frac{1}{|f'_{0,-}|}\circ f_{0,-}^{-1}) \cdot  (\varphi \circ f_{0,-}^{-1})
\|_{\HH^{\tilde\tau, p}}
 \, .
\end{align}
The first term in \eqref{eq:easypart} is estimated using
Lemmas~\ref{lem:multiplicationbysmooth},
\ref{lem:compositionbysmooth}, and~\ref{lem:differenceandcomposition}:
\begin{align*}
 & \|  \frac{\varphi}{|f'_{t,-}|} \circ f_{t,-}^{-1} - 
 \frac{\varphi}{|f'_{t,-}|} \circ f_{0,-}^{-1}
\|_{\HH^{\tilde\tau, p}}=
\|(\frac{\varphi}{|f'_{t,-}|} \circ f_{t,-}^{-1}\circ f_{0,-} - 
 \frac{\varphi}{|f'_{t,-}|} )\circ f_{0,-}^{-1}\|_{\HH^{\tilde\tau, p}}
  \\  &\qquad 
\le C \|\frac{\varphi}{|f'_{t,-}|} \circ f_{t,-}^{-1}\circ f_{0,-} - 
 \frac{\varphi}{|f'_{t,-}|} \|_{\HH^{\tilde\tau, p}}
\le C \| I - f_{t,-}^{-1}
  \circ f_{0,-} \|_{C^1}^{\tau - \tilde{\tau}}
 \|  \frac{\varphi}{|f'_{t,-}|}
  \|_{\HH^{\tau, p}}
\\
  &\qquad \leq C \| I - f_{t,-}^{-1}
  \circ f_{0,-} \|_{C^1}^{\tau - \tilde{\tau}} \| \varphi
  \|_{\HH^{\tau, p}}\, ,
\end{align*}
if $0 < \tilde{\tau} < \tau<1$. 
It is easy to check that  $\Vert f_{t,-} - f_{0,-} \Vert_{C^1} = O(|t|)$
implies that
$\| I - f_{t,-}^{-1}
  \circ f_{0,-} \|_{C^1}=
O(|t|)$. 

Note that $ \| \frac{1}{|f'_{t,-}|} \circ f_{0,-}^{-1} - 
 \frac{1}{|f'_{0,-}|}\circ f_{0,-}^{-1}  \|_{C^1} = O(|t|)$ as $t \to
0$. Since $0\le \tilde \tau <\tau <1$, for the second term in \eqref{eq:easypart} we then have by
Lemmas~\ref{lem:multiplicationbysmooth} and
\ref{lem:compositionbysmooth}  the upper bound
\begin{align*}
  \|(\frac{1}{|f'_{t,-}|} \circ f_{0,-}^{-1} - 
 \frac{1}{|f'_{0,-}|}\circ f_{0,-}^{-1}) 
\cdot (\varphi \circ f_{0,-}^{-1} )
  \|_{\HH^{\tilde{\tau}, p}} 
  & \leq C |t| \|
  \varphi \|_{\HH^{\tilde{\tau}, p}} \leq C |t| \| \varphi
  \|_{\HH^{\tau, p}}\, .
\end{align*}

We have shown that if $0 < \tilde{\tau} < \tau < \frac{1}{p}$ then
\[
\| 1_{I_0 \cap I_t} (\LL_t \varphi - \LL_0 \varphi)
\|_{\HH^{\tilde{\tau}, p}} \leq C |t|^{\tau - \tilde{\tau}} \|
\varphi \|_{\HH^{\tau, p}}\, .
\]

We return to \eqref{eq:easyandtrickypart} and consider $\|
1_{J_t} \LL_t \varphi \|_{\HH^{\tilde{\tau}, p}}$. 

If $q \in (p,\infty)$, letting $r$ be the 
conjugate of
$q/p$, then 
H\"{o}lder's inequality gives for any interval $J\subset I$
and any $u \in L^q(I)$ that
\[
\| 1_J u \|_{L^{p}(I)} = \| 1_J u^{p}
\|_{L^1((I)}^\frac{1}{p} \leq \bigl( \| 1_J \|_{L^r(I)}
\| u^{p} \|_{L^{q/p}(I)} \bigr)^\frac{1}{p} 
=|J|^\frac{1}{r p} \| u \|_{L^q(I)}
= |J|^{\frac{1}{p} - \frac{1}{q}} \| u \|_{L^q(I)}\, .
\]
Since $|J_t|\le C |t|$, it follows by the Sobolev embedding theorem
that,  taking $q>p$ such that $\tau=\frac{1}{p} -
\frac{1}{q}$,
\begin{align}\label{eq:embed_L_t}
  \| 1_{J_t} \LL_t \varphi \|_{L^{p}} \le |t|^{\frac1p -
    \frac1q} \Vert \LL_t \varphi \Vert_{L^q} \le C |t|^{\tau} \Vert \LL_t \varphi \Vert_{\HH^{\tau,p}}\le C |t|^{\tau} \Vert  \varphi \Vert_{\HH^{\tau,p}}
\, .
\end{align}
Also, \eqref{unifop} and Lemma~\ref{lem:Strichartz} give, for $\tau < 1/p$,
\begin{align}\label{unifop'}
  \| 1_{J_t} \LL_t \varphi \|_{{\HH^{\tau,p}}} \le  C  \Vert \varphi \Vert_{\HH^{\tau,p}}
\, .
\end{align}
Finally, since $\HH^{\tilde \tau,p}=[L^p, \HH^{\tau,p}]_{\tilde\tau/\tau}$
(where $[\BB_0,\BB_1]_\theta$ denotes complex interpolation) interpolating at $\theta=\tilde \tau/\tau$
(see \cite[\S 2.5.2]{RS} and \cite[\S 1.9]{TrFA}) between \eqref{eq:embed_L_t} and \eqref{unifop'}, we get
\begin{align*}
  \Vert 1_{J_t} \LL_t \varphi \Vert_{\HH^{\tilde\tau,p }} &\le 
 C
   |t|^{\tau (1-\theta)}
  \Vert \varphi \Vert_{\HH^{\tau,p }}=C
   |t|^{\tau - \tilde\tau}
  \Vert \varphi \Vert_{\HH^{\tau,p }} \, .
\end{align*}

In the same way, since $|K_t|\le C|t|$, we get 
$
\| 1_{K_t} \LL_0 \varphi \|_{\HH^{\tilde{\tau},p}} \leq C
|t|^{\tau - \tilde{\tau}} \| \varphi \|_{\HH^{\tau, p}}$.
Recalling \eqref{eq:easyandtrickypart}, we have
proved  
$
\| \LL_t \varphi - \LL_0 \varphi
\|_{\HH^{\tilde{\tau},p}} \leq C |t|^{\tau -
  \tilde{\tau}} \| \varphi \|_{\HH^{\tau,p}}$
for $0 < \tilde{\tau} < \tau < \frac{1}{p}$.
\end{proof}


\section{Proofs of Theorems~\ref{thm:susc_holom}, \ref{newth} and Corollaries~\ref{lecor}, \ref{cor:realanalytic}}
\label{pr}

\subsection{Proof of Theorem~\ref{thm:susc_holom} on the fractional susceptibility function}
Clearly,
\begin{equation} \label{b0}
   \int_I (\LL_t - \LL_0) \varphi(x) \, \D
     x = 0\, , \,\,\,
\forall  \varphi \in L^1\, .
     \end{equation}

Since $I$ is compact, we may assume  that $q\ne \infty$ (otherwise replace $q$ 
by any finite number larger than $(1-\eta)^{-1}$).  It suffices to consider the
contribution of the improper integral
$\int_0^\infty$ in $\Psi_\phi(\eta,z)$, the computation for the
integral $\int_{-\infty}^0$ is exactly the same.
We will prove claim \eqref{main} for $\Psi_\phi(\eta, z)$, the proof
for $\Psi^\mathrm{fr}_\phi(\eta, z)$ is obtained by a slight simplification.
  Let $q'> 1$ be the conjugate of $q$, that is
 $  \frac1q + \frac{1}{q'} = 1$.
  Then,  by H{\"o}lder's inequality,
  \[
  \int_I \biggl| \phi \cdot \LL_t^k (\LL_t -
    \LL_0)\rho_0  \biggr|
  \, \D x \leq  \| \LL_t^k
  (\LL_t - \LL_0) \rho_0 \|_{L^{q'}} \| \phi
  \|_{L^q}.
  \]
 Next, fixing $p \in (1, q')$ (we shall need to take $p$ close to $1$ 
 soon) and setting
    \begin{align}  
    \label{deftt}
      \tilde \tau :=\frac{1}{p} -\frac{1}{q'} =\frac{1}{p} -1+ \frac{1}{q} \in \bigl (0,   \frac{1}{p}-\eta \bigr  ) \, ,
    \end{align}
  the Sobolev embedding theorem
  \cite[Theorem~1.3.5]{GrafakosModern}, gives a constant $C_{p,q'}$ such
  that 
  \begin{align*}
  \int_I \biggl| \phi \cdot \LL_t^k (\LL_t -
    \LL_0)\rho_0   \biggr| \, \D x 
  &\leq  C_{p,q'} \| \LL_t^k
  (\LL_t - \LL_0) \rho_0 \|_{\HH^{\tilde{\tau},p}}
  \| \phi \|_{L^q} \, .
  \end{align*}
  Thus, for each fixed $k\ge 0$, by Lemma~\ref{lem:Lt-L0}, the improper
  integral defining the coefficient of $z^k$ in $\Psi(\eta,z)$ or $\Psi^\mathrm{fr}(\eta, z)$  
  is a well-defined complex number. 
 
  \smallskip
 
Assume from now on that $p\in (1,\min( p_0,q'))$
where $p_0$ is from Lemma~\ref{lem:lasota_yorke}, and let $\kappa$ be as in Definition~\ref{defkappa}.

Since $\Lambda^{- \tilde \tau - 1 +   \frac1p}=\Lambda^{- \frac1q}$, 
  using \eqref{b0} for $\varphi=\rho_0$, we have 
for any $\kappa_1 > \max (\kappa, \Lambda^{- \frac1q})$, by Keller and
Liverani's consequence \eqref{eq:decayestimate_t} of the uniform Lasota--Yorke and perturbation bounds,  constants $\ve_1>0$ and\footnote{As $\eta\to 1$, we have $C_\eta\le C (1-\eta)^{-1}$.}  $C_\eta<\infty$ such that
  \begin{equation}\label{decc}
  \| \LL_t^k (\LL_t - \LL_0) \rho_0
  \|_{\HH^{\tilde{\tau},p}} \leq C_\eta \kappa_1^k \|
  (\LL_t - \LL_0) \rho_0 \|_{\HH^{\tilde{\tau},p}}\, , \,
\forall |t|\le  \ve_1\, .
  \end{equation}
     
  For any $\tau\in (\tilde \tau, \frac{1}{p})$, the bound
  \eqref{the:Lt-L0} implies
  \begin{equation}\label{perturb}
    \| (\LL_t - \LL_0) \rho_0
    \|_{\HH^{\tilde{\tau},p}} \leq C |t|^{\tau - \tilde{\tau}}
    \| \rho_0 \|_{\HH^{\tau,p}} \, , \,
\forall |t|\le \ve_1\, .
  \end{equation}
  We conclude that for any $0\le |t|<\ve_1$
  \begin{equation}\label{gg}
  \int_I \biggl| \phi \cdot  \LL_t^k \frac{(\LL_t -
    \LL_0)\rho_0}{|t|^{1+\eta}}   \biggr|
  \, \D x \leq C_\eta \kappa_1^k |t|^{-1-\eta + \tau -
    \tilde{\tau}} \| \rho_0 \|_{\HH^{\tau,p}} \| \phi
  \|_{L^q} \, .
  \end{equation}
   Since $q>(1-\eta)^{-1}$, we may choose $\tau<1/p$ and $\tilde{\tau}=1/p-1+1/q$
  such that 
  $
  \eta < \tau - \tilde{\tau}$. For such
   $\tau$ and $\tilde{\tau}$, we may integrate \eqref{gg} over $t\in(0,\ve_1)$, and we get
  \begin{equation}\label{b1}
  \int_{0}^{\ve_1} \int_I \biggl|\phi(x)  \LL_t^k
  \frac{(\LL_t - \LL_0)\rho_0}{t^{1+\eta}}   \biggr| \, \D x \, \D t \leq C
  \kappa_1^k \| \rho_0 \|_{\HH^{\tau,p}} \| \phi
  \|_{L^q} \, .
  \end{equation}

  For $t > \ve_1$, we have $\LL_t = \LL_{\pm
    \ve_1}$,   so that, using again  the arguments above,
  \begin{align}
  \nonumber
     \int_{t> \ve_1} \int_I \biggl|\phi \cdot  \LL_t^k \frac{(\LL_t -
      \LL_0)\rho_0}{t^{1+\eta}} \biggr|
    \, \D x \, \D t   
    &\le \int_{t> \ve_1} \frac{1}{t^{1+\eta}} \, \D t 
    \cdot  \int_I \biggl| \phi \cdot \LL_{\pm\ve_1}^k
      (\LL_{\ve_1} - \LL_0)\rho_0 \biggr| \, \D x  \\
      \label{b2}
  &\leq C_\eta \kappa_{1}^k 
    \ve_1^{\tau - \tilde{\tau}-\eta} \| \rho_0 \|_{\HH^{\tau,p}}
    \| \phi \|_{L^q}
    \, .
    \end{align}
  By Fubini, the bounds \eqref{b1} and \eqref{b2} imply that
  $\Psi_\phi(\eta,z)$ and  $\Psi^\mathrm{fr}_\phi(\eta,z)$
are holomorphic in the disc
  of radius $\kappa_1^{-1}$, if $\phi \in L^q$ for $q >(1-\eta)^{-1}$.
Since we may take $q$ arbitrarily close to $(1-\eta)^{-1}$
and $\kappa_1$ arbitrarily close to $\min (\kappa, \Lambda^{-\frac 1 q})$, this shows 
\eqref{main}.

\smallskip

 We next show claim \eqref{eq:response} of the theorem.  The bounds \eqref{b0}, \eqref{decc},  \eqref{b1}, and \eqref{b2} 
  together with Fubini and the dominated convergence theorem imply that, as holomorphic functions
  in the 
  disc of radius $\kappa_1^{-1}> 1$,
  \begin{align*}
  \Psi_\phi(\eta,z)&= \frac{\Gamma_\eta}{2}
    \sum_{k=0}^\infty z^k \int_{-\infty}^\infty  \int_I \phi  \biggl(
    \LL_t^k \frac{(\LL_t - \LL_0)
      \rho_0}{|t|^{1+\eta}} \biggr) \sign(t)  \,
    \D x \, \D t \\ &= \frac{\Gamma_\eta}{2}
    \int_{-\infty}^{\infty}
    \int_I \phi \biggl( (I - z \LL_t)^{-1}
    \frac{(\LL_t - \LL_0) \rho_0}{|t|^{1+\eta}}
     \biggr) \sign(t)  \, \D x \, \D t \, .
  \end{align*}
  In view of  \eqref{b0}, the trivial identity
  $
    ( \LL_t - \LL_0 ) \rho_0 = ( I - \LL_t
    )(\rho_t - \rho_0) $ implies the key formula \eqref{key}.
  In particular,
  \begin{align*}
    \Psi_\phi (\eta, 1)
    &=\frac{\Gamma_\eta}{2}
    \int_{-\infty}^{\infty}\frac{ 1 }{|t|^{1+\eta}}\int_I
    \phi(x)   (\rho_0(x) - \rho_t(x)  )  \, \D x \, \sign(t)\, 
    \D t \\
    &=  M^{\eta}\left( \int_I \phi(x) \rho_t(x)
    \, \D x \right) \bigg|_{t=0} \, .
  \end{align*}
   This proves claim \eqref{eq:response}.

\medskip  
To conclude the proof of Theorem~\ref{thm:susc_holom}, it only remains to prove \eqref{eq:last}. Note that $\LL_0^k$ does not depend on $t$
and acts on functions depending on $x$. Since we have shown that every improper integral defining  $\Psi^\mathrm{fr}_\phi(\eta,z)$
is convergent, and since the sum over $k$ in \eqref{deffr} converges absolutely for $|z|\le 1$, we may
write
\begin{align*}
  \Psi^{\mathrm{fr}}_\phi(\eta,z)
&= \sum_{k=0}^\infty  \int_I \phi
z^k \LL_0^k  \biggl ( \frac{\Gamma_\eta}{2}
     \int_{-\infty}^\infty \biggl(
      \frac{(\LL_0 - \LL_t)
       \rho_0}{|t|^{1+\eta}} \biggr) \sign(t) \, \D t  \biggr )  \,
     \D x \\
&= \int_I \phi
(I-z \LL_0)^{-1}  \biggl ( \frac{\Gamma_\eta}{2}
     \int_{-\infty}^\infty \biggl(
      \frac{(\LL_0 - \LL_t)
       \rho_0}{|t|^{1+\eta}} \biggr)  \sign(t) \, \D t \biggr )  \,
     \D x  \, .
\end{align*}
Thus, \eqref{eq:last} just follows from the definition of the Marchaud derivative
$M^\eta$. \qed

\subsection{Proof of Theorem~\ref{newth} on the generalized fractional susceptibility function} \label{sec:generalizedproof}

The proof of Theorem~\ref{newth} is the same as the proof of
Theorem~\ref{thm:susc_holom} up to equation \eqref{gg}. Instead of
\eqref{gg} we have
\[
  \int_I \biggl| \phi \cdot \LL_t^k \frac{(\LL_t -
    \LL_0)\rho_0}{\ell(|t|)} \biggr| \, \D x \leq C \kappa_1^k
  \frac{|t|^{\tau - \tilde{\tau}}}{\ell (|t|)} \| \rho_0
    \|_{\HH^{\tau,p}} \| \phi \|_{L^q}\, .
\]
We choose $\tau < 1/p$ and $\tilde{\tau} = 1/p - 1 + 1/q$. Then $\tau
- \tilde{\tau} = 1 - 1/q \geq \gamma$ and
\[
\int_{0}^{\varepsilon_1} \frac{t^{\tau -
    \tilde{\tau}}}{\ell (t)} \, \D t \leq
\int_{0}^{\varepsilon_1} \frac{t^{\gamma}}{\ell (t)}
\, \D t < \infty\, , 
\]
by the assumption on $\gamma$. Hence we get, similar to \eqref{b1},
that
\[
\int_{0}^{\ve_1} \int_I \biggl|\phi(x) \LL_t^k \frac{(\LL_t -
  \LL_0)\rho_0}{\ell(t)} \biggr| \, \D x \, \D t \leq C
\kappa_1^k \| \rho_0 \|_{\HH^{\tau,p}} \| \phi \|_{L^q} \, .
\]

Similar to \eqref{b2}, we have
\begin{align*}
  \int_{t> \ve_1} \int_I \biggl|\phi \cdot \LL_t^k \frac{(\LL_t -
    \LL_0)\rho_0}{\ell(t)} \biggr| \, \D x \, \D t &\le \int_{t> \ve_1}
  \frac{1}{\ell(t)} \, \D t \cdot  \int_I \biggl| \phi
  \cdot \LL_{\ve_1}^k (\LL_{ \ve_1} - \LL_0)\rho_0 \biggr| \, \D
  x \\ &\leq C \kappa_{1}^k  \| \rho_0
  \|_{\HH^{\tau,p}} \| \phi \|_{L^q} \, ,
\end{align*}
where we used the assumption $\int_0^1 1/ \ell(t) \, \D t < \infty$.
We may choose $\kappa_1$ arbitrary close to $\min (\kappa,
\Lambda^{-\frac{1}{q}})$, which proves (\ref{maingeneralized}).

Claims \eqref{responsegeneralized} and \eqref{lastgeneralized} are
proved just as the corresponding claims in
Theorem~\ref{thm:susc_holom},  replacing $|t|^{1+ \eta}$ by
$\ell(t)$ in the proofs.

\subsection{Proof of Corollary~\ref{cor:realanalytic} on holomorphic extensions} \label{sec:realanalyticproof}

\begin{remark}[Sketch of an alternative proof using Morera's theorem]
The proof below is by estimating the growth of derivatives.  We sketch here a more conceptual proof: 
First extend the definition of $M^{\eta}$ and
Theorem~\ref{thm:susc_holom} to complex $\eta$ with $0<\Re \eta < 1 - q^{-1}$. 
This implies in particular that  the double integrals appearing in $\Psi_\phi(\eta,z)$
are well-defined for such $\eta$, and Fubini is justified.
Then, since Euler's Gamma function and $t^{-1-\eta}$ are holomorphic in the domain
considered, Morera's theorem gives the desired holomorphic extension.
\end{remark}

We first show that the function 
$$\eta \mapsto M^\eta \bigl( \int_I \phi (x)
\rho_t (x) \, \D t \bigr) \big|_{t=0}$$ 
is holomorphic in the
strip $0 < \Re \eta < 1 - \frac{1}{q}$. 
We give the argument for the right-sided derivative, the left-sided case is the same
up to introducing signs.
Let $0 < \gamma < 1
- \frac{1}{q}$ 
and take $\eta_0$ such that $0 < \Re \eta_0 <
\gamma$. The right-sided fractional derivative with the parameter $\eta = \eta_0 +
\zeta$ is given by
\begin{align*}
  F_{\eta_0} (\zeta) :=& \, M^{\eta_0 + \zeta}_- \biggl(\int_I \phi (x) \rho_t (x) \, \D t
  \biggr)\bigg|_{t=0} 
  = \Psi_\phi^- (\eta_0 + \zeta, 1) \\ 
  =& \,
  \sum_{k=0}^\infty \frac{\Gamma_{\eta_0+\zeta}}{2} \int_I \phi
  \int_{0}^\infty \biggl( \LL_t^k \frac{(\LL_t - \LL_0)
    \rho_0}{t^{1 + \eta_0 + \zeta}} \biggr)  \, \D
  t \, \D x \, .
\end{align*}
Let $\ell_n (t) = |t|^{1 + \eta_0} (\log |t|)^{-n}$. Using
Theorem~\ref{newth} and  $|t|^{-\zeta} =
\sum_{n=0}^\infty \frac{1}{n!} (\log |t|)^n \zeta^n$, we then have
\begin{align*}
  F_{\eta_0} (\zeta) &= \sum_{k=0}^\infty
  \frac{\Gamma_{\eta_0+\zeta}}{2} \int_I \phi 
  \int_{0}^\infty
  \biggl( \LL_t^k (\LL_t - \LL_0) \rho_0
   \sum_{n=0}^\infty \frac{(\log  t)^n}{t^{1 + \eta_0}} 
   \frac{\zeta^n}{n!} \biggr)  \,
  \D t \, \D x \\ 
  &= \frac{\Gamma_{\eta_0+\zeta}}{2}
  \sum_{k=0}^\infty \sum_{n=0}^\infty \int_I \phi
  \int_{0}^\infty \biggl( \LL_t^k (\LL_t - \LL_0) \rho_0
  \frac{(\log t)^n}{t^{1 + \eta_0}} \frac{\zeta^n}{n!} \biggr)
  \, \D t \, \D x \\ &=
  \frac{\Gamma_{\eta_0+\zeta}}{2} \sum_{n=0}^\infty \frac{\Psi_\phi^-
    ((\ell_n), 1)}{n!}  \zeta^n\, .
\end{align*}
To justify exchanging the  order of sums and integrations above, we will
estimate below
\begin{equation}\label{estim}
\frac{\zeta^n}{n!} \int_I \phi \int_{-\infty}^\infty \biggl( \LL_t^k (\LL_t -
\LL_0) \rho_0 \frac{(\log |t|)^n}{|t|^{1 + \eta_0}} \biggr) \sign(t)
\, \D t \, \D x \, .
\end{equation}
Since $F_{\eta_0} (\zeta)$ is a product of $\frac{\Gamma_{\eta_0+\zeta}}{2}$
and $\sum_{n=0}^\infty \frac{\Psi_\phi ((\ell_n), 1)}{n!} \zeta^n$, 
which are both holomorphic in $\zeta$, provided $|\zeta|$ is
sufficiently small, this will show that $F_{\eta_0} (\zeta)$ is holomorphic.

To estimate \eqref{estim},  we first use the proof of Theorem~\ref{newth}, to get
\[
\biggl| \frac{\zeta^n}{n!} \int_I \phi \int_{-\infty}^\infty \biggl(
\LL_t^k (\LL_t - \LL_0) \rho_0 \frac{(\log |t|)^n}{|t|^{1 + \eta_0}}
\biggr) \sign(t)  \, \D t \, \D x \biggr| \leq C_n
\kappa_1^k \zeta^n \lVert \rho_0 \rVert_{\mathcal{H}^{\tau,p}} \lVert
\phi \rVert_{L^q}\, ,
\]
where
\[
C_n = \frac{C}{n!} \biggl(\int_{0}^{\varepsilon_1}
\frac{t^{\gamma}}{|\ell_n (t)|} \, \D t + 
\int_{t >
  \varepsilon_1} \frac{1}{|\ell_n (t)|} \, \D t \biggr)\, ,
\]
and $C$ is a constant that does not depend on $n$.
Next, since $\gamma > \Re \eta_0$ we have
\[
\biggl| \int_{-\varepsilon_1}^{\varepsilon_1}
\frac{|t|^{\gamma}}{\ell_n (|t|)} \, \D t \biggr| = 
2
\int_0^{\varepsilon_1} \frac{|\log t|^n}
{t^{1 - (\gamma - \Re
    \eta_0)}} \, \D t \leq 2 \int_0^1 
    \frac{|\log t|^n}{t^{1 -
    (\gamma - \Re \eta_0)}} \, \D t = 2 \frac{n!}{(\gamma
  - \Re \eta_0)^{n+1}}
\]
and
\begin{align*}
\biggl| \int_{|t| > \varepsilon_1} \frac{1}{\ell_n (|t|)} \,
\D t \biggr| &\leq 2 |\varepsilon_1|^{1 + \Re \eta_0}
|\log \varepsilon_1|^n + 2 \int_{1}^\infty \frac{(\log t)^n}{t^{1 +
    \Re \eta_0}} \, \D t \\ &
    = 2 |\varepsilon_1|^{1 +
  \Re \eta_0} |\log \varepsilon_1|^n + 2 \frac{n!}{(\Re
  \eta_0)^{n+1}}\, .
\end{align*}
In conclusion, $C_n$ grows at most with an exponential speed. Hence,
we may change order of integration if $\zeta$ is in a sufficiently
small disc. In particular, our estimates show that $\frac{\Psi_\phi
  ((\ell_n), 1)}{n!}$ grows at most exponentially in $n$, so that
\[
  F_{\eta_0} (\zeta) = \frac{\Gamma_{\eta_0+\zeta}}{2}
  \sum_{n=0}^\infty \frac{\Psi_\phi ((\ell_n), 1)}{n!}  \zeta^n\, ,
\]
holds when $\zeta$ is in a disc around the origin.

The proof that the function $(\eta, \zeta) \mapsto \Psi_\phi (\eta,
z)$ is holomorphic in the domain 
$$\{\, (\eta, z) \mid 0 < \Re \eta
< 1-\frac{1}{q} ,\, \ |z| < \min (\kappa, \Lambda^{-\frac{1}{q}}) \,\}$$
 uses the same
estimates, writing instead
\begin{align*}
\Psi_\phi (\eta_0 + \zeta, z) &= \sum_{k=0}^\infty
\frac{\Gamma_{\eta_0+\zeta}}{2} z^k \int_I \phi \int_{-\infty}^\infty
\biggl( \LL_t^k \frac{(\LL_t - \LL_0) \rho_0}{|t|^{1 + \eta_0 +
    \zeta}} \biggr) \sign(t) \, \D t \, \D x
\\ &= \sum_{k=0}^\infty \frac{\Gamma_{\eta_0+\zeta}}{2} z^k \int_I \phi
\int_{-\infty}^\infty \biggl( \LL_t^k (\LL_t - \LL_0) \rho_0
\sum_{n=0}^\infty \frac{(\log |t|)^n}{|t|^{1 + \eta_0}}
\frac{\zeta^n}{n!} \biggr) \sign(t)  \, \D t \, \D x \, .
\end{align*}

\subsection{Proof of Corollary~\ref{lecor} on horizontal perturbations}
\label{proofcor}

By \cite[Theorem~2.8]{BS} (see \cite{BS09}), the horizontality assumption
implies that $f_t$ is tangential to the topological class of $f_0$ that is,
there exist a $C^2$ perturbation $(\tilde f_t)$ of $f_0$ with $|\tilde
f_t-f_t|=O(t^2)$ (as $|t|\to 0$) and homeomorphisms $h_t$ of $I$ with
$h_t(c)=c$ and $\tilde f_t = h_t \circ f \circ h_t^{-1}$.  Note  that, letting $\tilde \LL_t$ be
the operator associated to $\tilde f_t$, the proof of
Lemma~\ref{lem:Lt-L0} gives $\ve_2$ such that (cf.  \cite[(27) proof
  of Proposition~3.3]{BS}) that
\[
\|(\LL_t - \tilde \LL_t)\varphi \|_{\HH^{\tilde{\tau},p}}\le C
|t|^{2(\tau-\tilde \tau)}\|\varphi\|_{\HH^{\tau,p}} \, , \quad \forall
|t|< \ve_2 \, ,
\]
(this
will be used to show \eqref{noper}).
Since we may choose $0<\tilde \tau<\tau <1/p$ such that $2(\tau-\tilde \tau)>1$, it is not hard to see that
we may replace $f_t$ by $\tilde f_t$ in the proof of Corollary~\ref{lecor}.

The  two first equalities claimed in 
Corollary~\ref{lecor} follow from   statement (b) of
Theorem~\ref{thm:susc_holom}, combined with Lemma~\ref{limit}
and  \cite[Proposition~4.3, Theorem.~5.1]{BS}, respectively.

The last statement in the corollary, \eqref{noper}, can be
deduced\footnote{It would be interesting to have a more direct proof
  of \eqref{noper}.} from statement (c) of
Theorem~\ref{thm:susc_holom} together with the following claim (the
last term in the right-hand side is understood in the sense of
distributions, integrating against a continuous function)
\begin{equation}\label{desi}
\lim_{\eta \uparrow 1} M^\eta (\LL_t \rho_0)|_{t=0}
=-X_0' \rho_0 - X_0 (\rho_0^\mathrm{reg})' -  \sum_{k=1}^{\infty} \bar s_ k X_0(c_k) \delta_{c_k}  \, .
\end{equation}
Indeed, it suffices to note that
(recall that $c$ is not periodic for $f$)
\[
\int \phi \sum_{k=0}^\infty \sum_{j=1}^{\infty}\LL_0^k( \bar s_j
X_0(c_j) \delta_{c_j})= \sum_{k=0}^\infty\sum_{j=1}^{\infty} \bar s_j
X_0(c_j) \phi(c_{j+k})= \sum_{\ell=0}^{\infty} \phi_{c_\ell}
\sum_{j=1}^{\ell} \bar s_j X_0(c_j) \, .
\]

The desired identity \eqref{desi} will be an immediate consequence of the decomposition $\rho=\rho^\mathrm{reg}+\rho^\mathrm{sal}$, with
\eqref{rsal}, and Lemma~\ref{limit} 
after we
establish that $t \mapsto \int_I \phi(x)  (\LL_t \rho_0) (x) \, \D x$ is
differentiable at $t=0$ for any $\phi \in C^0$, with derivative
\begin{equation}\label{esta}
\partial_t \biggl ( \int_I \phi(x)  (\LL_t \rho_0) (x) \, \D x
\biggr )\biggr |_{t=0}=-
\int_I \phi \,   (X_0 \rho_0)' \, \D x \, .
\end{equation}

To show \eqref{esta}, we first recall Step~2 in the proof of
\cite[Theorem~5.1]{BS}.  For this, letting $c_{k,t}=f_t^k(c)$ and
denoting by $\BB_t$ the vector space of functions
$\varphi=\varphi^\mathrm{reg} +\sum_{k=1}^{\infty} a_k H_{c_{k,t}}$ where
$\varphi^\mathrm{reg}\in BV\cap C^0$ is supported in $I$ and is
differentiable with derivative in $BV$, and the $a_k$ are complex
numbers, we recall the invertible map $G_t \colon \BB_t \to \BB_0$
from \cite{BS}, defined by
\[
G_t(\varphi)= \varphi^\mathrm{reg} +\sum_{k= 1}^{\infty} a_k H_{c_{k}} \, .
\]
In this notation, Step~2 in the proof of
\cite[Theorem~5.1]{BS} (the proof uses $f_t=\tilde f_t$) says 
\begin{equation}\label{done}
  \partial_t \bigl ( G_t( \LL_t (G_t^{-1} \rho_0)) \bigr )
  |_{t=0} = -X_0' (\rho_0) - X_0 (\rho_0^\mathrm{reg})' \in BV \, .
\end{equation}
To use the above identity, we decompose $\LL_t \rho_0 - \rho_0$ into
\begin{align}\label{above}
(G_t( \LL_t (G_t^{-1} \rho_0))-\rho_0)
 + (\LL_t(G_t^{-1}(\rho_0)) -G_t( \LL_t (G_t^{-1} \rho_0)))
+(\LL_t \rho_0  - \LL_t(G_t^{-1}(\rho_0))) \, .
\end{align}
The first term above is handled by \eqref{done}. For the other two terms, recall that $c$ is not periodic, so $P_f=\infty$.
For the last term in  \eqref{above},  
note that, since $c_{k,t}=h_t(c_k)$, we have
(see \cite[Lemma~2.2, Remark~2.5]{BS})
\begin{equation}\label{TCEs}
\lim_{t \to 0}\frac{c_{k,t}-c_k}{t}
=-\sum_{j=0}^{\infty} \frac{X_0(c_{j+k})}{(f^{j+1})' (c_k)}=  \alpha(c_k)
\, , \quad \forall k \ge 1\, .
\end{equation}
Therefore, since $\LL_0 (\delta_{c_k})= \delta_{c_{k+1}}$, using bounded distortion
\begin{align}
  \nonumber \lim_{t \to 0} \frac{\int \phi (\LL_t \rho_0 -
    \LL_t(G_t^{-1}(\rho_0)))\, \D x}{t} &=\lim_{t \to
    0}\frac{\int \phi \LL_t (\sum_{k = 1}^{\infty } \bar s_k
    H_{c_k}- \sum_{k = 1}^{\infty } \bar s_k H_{c_{k,t}}) \, \D
    x}{t} \\
  \label{third} &= \sum_{k = 1}^{\infty}\bar s_k \alpha (c_k) \phi(c_{k+1}) 
  = \sum_{k = 1}^{\infty} \bar s_{k+1} f'(c_k) \alpha (c_k)
  \phi(c_{k+1})\, .
\end{align}
Finally, for the second term in \eqref{above}, we have
\begin{align*}
  &(\LL_t(G_t^{-1}(\rho_0)) -G_t( \LL_t (G_t^{-1} \rho_0)))
  \\ &\qquad= [\LL_t(\rho_0^\mathrm{reg}+\sum_{k =1}^{\infty}\bar
    s_k H_{c_{k,t}} ))]^\mathrm{sal} -G_t
     [\LL_t(\rho_0^\mathrm{reg}+\sum_{k = 1}^{\infty} \bar s_k
       H_{c_{k,t}} ))]^\mathrm{sal}\\ &\qquad= \biggl
     (\frac{1}{|f'_{t,-}(c)|}+\frac{1}{|f'_{t,+}(c)|}\biggr )
     \rho_0^\mathrm{reg}(c) (H_{c_{1,t}}-H_{c_1}) + \sum_{k =
       1}^{\infty} \frac{ \bar s_k} {|f'(c_{k,t})|}
     (H_{c_{k+1},t}- H_{c_{k+1}})\, .
\end{align*}
Therefore, using $\LL_0 \rho_0(c_1)=\rho_0(c_1)=-\bar s_1$ and $\rho_0^\mathrm{sal}(c)=0$,
\begin{align}\label{second}
\lim_{t \to 0} &\int \phi \frac{(\LL_t(G_t^{-1}(\rho_0)) -G_t( \LL_t (G_t^{-1} \rho_0)))}
{t} \, \D x \\
&\qquad\qquad= -\bar s_1 X_0(c_1)\phi(c_1) - \sum_{k = 1}^{\infty } \bar s_{k+1} \alpha (c_{k+1}) \phi(c_{k+1}) \, .
\end{align}
The twisted cohomological equation \eqref{TCE} implies that
the sum of \eqref{third} and \eqref{second} is
\begin{align*}
-\bar s_1 X_0(c_1) \phi(c_1)&+\sum_{k = 1}^{\infty} \bar s_{k+1} \bigl (f'(c_k) \alpha (c_k) -\alpha (f(c_{k})\bigr )\phi(c_{k+1})
\\
&\qquad \qquad= - \sum_{k = 0}^{\infty} \bar s_{k+1} X_0(c_{k+1}) \phi(c_{k+1})
=-\int \phi  (X_0 \rho_0^\mathrm{sal})'\, .
\end{align*}
This concludes the proof of \eqref{esta}, and of Corollary~\ref{lecor}.
\qed


\appendix

\section{Proof of Lemma~\ref{limit}}
\label{ll}

Without loss of generality, we set  $t_0=0$.
 Since $g$ is bounded, there exists  $A$ such that 
  \[
  \biggl| \int_1^\infty \frac{g(0) - g(-t)}{t^{1+\eta}} \,
  \D t \biggr| \leq  \int_1^\infty \frac{2A}{t^{1+\eta}} \,
  \D t < \infty \, .
  \]
  Since $g$ is differentiable at $0$, there exists  $B$ such
  that 
  \[
  \biggl| \int_0^1 \frac{g(0) - g(-t)}{t^{1+\eta}} \, \D t
  \biggr| \leq \int_0^1 \frac{B}{t^\eta} \, \D t < \infty \, .
  \]
  This shows that $M_+^\eta g (0)$ exists for all $\eta < 1$.
Exchanging $g$ with $t \mapsto g(-t)$, we get that $M_-^\eta g
  (0)$ exists for all $\eta < 1$.
  
  Let $\varepsilon > 0$. We will prove 
  $
  g'(0) - \varepsilon \leq \liminf_{\eta \uparrow 1} M_+^\eta g (0)
  \leq \limsup_{\eta \uparrow 1} M_+^\eta g (0) \leq g'(0)+
  \varepsilon$.
  Since $\varepsilon$ is arbitrary, this implies that
  $
  \lim_{\eta \uparrow 1} M_+^\eta g (0) = g'(0)$.

  Since the limit
  $
  g'(0) = \lim_{h \to 0} \frac{g(h) - g(0)}{h}
  $
  exists, we can find  $\delta > 0$ such that
  \[
  - \varepsilon \leq \frac{g(h) - g(0)}{h} - g'(0) \leq \varepsilon
\, , \, \, \, \, \forall |h| \leq \delta \, .
  \]
Next, since $0<\eta <1$, we can  write
  \begin{align*}
    M_+^\eta g (0) 
 &=
    \frac{\eta}{\Gamma (1-\eta)} \biggl( \int_0^\delta \frac{g (0)
      - g(-t)}{t} \frac{1}{t^\eta} \, \D t +
    \int_\delta^\infty \frac{g (0) - g(-t)}{t^{1+\eta}} \,
    \D t \biggr) \\ &\leq \frac{\eta}{\Gamma (1-\eta)}
    \biggl( \int_0^\delta (g'(0) + \varepsilon) \frac{1}{t^\eta} \,
    \D t + \int_\delta^\infty \frac{g(0) -
      g(-t)}{t^{1+\eta}} \,\D t \biggr) \\ & \leq
    \frac{\eta}{(1-\eta) \Gamma(1-\eta)} \delta^{1-\eta}
    (g'(0) + \varepsilon) + \frac{\eta}{\Gamma(1-\eta)}
    \int_\delta^\infty \frac{2A}{t^{1+\eta}} \, \D t \\ &=
    \frac{\eta}{\Gamma(2-\eta)} \delta^{1-\eta}
    (g'(0)+\varepsilon) + \frac{1}{\Gamma(1-\eta)} 2A
    \delta^{-\eta} \, .
  \end{align*}
  Using that $\Gamma (1) = 1$, $\lim_{\eta \uparrow 1} \Gamma(1-\eta) =
  \infty$ and $\lim_{\eta \uparrow 1} \delta^{1-\eta} = 1$, this shows
  that
  \[
  \limsup_{\eta \uparrow 1} M_+^\eta g (0) \leq g'(0) + \varepsilon \, .
  \]
  In the same way, we get
  \[
  M_+^\eta g (0) \geq \frac{\eta}{\Gamma(2-\eta)}
  \delta^{1-\eta} (g'(0)-\varepsilon) - \frac{1}{\Gamma(1-\eta)}
  2A \delta^{-\eta} \, ,
  \]
  and thus
  $                                                                    
  \liminf_{\eta \uparrow 1} M_+^\eta g (0) \geq g'(0) - \varepsilon 
  $.
  Using $t \mapsto g(-t)$, we get 
  $\lim_{\eta \uparrow 1} M_-^\eta g (0) = - g'(0)$.


\section{Proof of the uniform Lasota--Yorke estimate (Lemma~\ref{lem:lasota_yorke})}
\label{AA}

Note that in view of \eqref{unifop} it is enough to show that for any
$\tilde \Lambda < \Lambda$, there exist
 $p_0\in (1, p(f))$ and $\ve_0\le \ve$ such that for any 
    $p \in (1, p_0)$ and  $0 \le \tilde\tau < \tau < \frac1p$
  there exist $N \ge 1$ and a  constant $C(N)<\infty$ 
   such that
\begin{align}\label{eq:lasota_yorke'}
  \| \LL_t^N \varphi \|_{\HH^{\tau,p}} \leq 
   \tilde \Lambda^{(- \tau - 1 +   \frac1p)N}
  \| \varphi \|_{\HH^{\tau,p}} + C(N) \| \varphi
  \|_{\HH^{\tilde\tau,p}} \, , \quad \forall\,  |t|<\ve_0\, , \, \, \forall \, \varphi\, .
\end{align}
(To conclude for general $n=jN+k$, use a geometric series for $jN$  and \eqref{unifop} for $k\le N-1$.)

Our strategy to prove
\eqref{eq:lasota_yorke'}  is to follow Thomine's argument
\cite{Thomine}, using a ``zoom'':\footnote{The zoom and
  localisation can perhaps  be replaced by the fragmentation and
  reconstitution lemmas in \cite[Chapter 2]{BaladiZeta}.}  For this,
note that Lemma~\ref{lem:compositionbysmooth} gives $C_\mathrm{zoom} =
C(\tau,p) > 0$ such that for any positive number $Z>0$, setting
$\ZZ(x) = Z \cdot x$,
  \begin{align}\label{eq:sobolev_scaling}
    \Vert \varphi \circ \ZZ \Vert_{\HH^{\tau,p}} \le C_\mathrm{zoom}
    Z^{\tau - \frac1p} \Vert \varphi \Vert_{\HH^{\tau,p}} +
    C_\mathrm{zoom} Z^{-\frac1p} \Vert \varphi \Vert_{L^p} \, ,
    \forall n \, .
  \end{align}
Define $\ZZ_n \colon \bR \to \bR$ by $\ZZ_n(x) =
  Z_n \cdot x$, where\footnote{It is essential that $Z_n$ will be chosen uniformly in $t$ for each fixed $n$.} the sequence of zooms $Z_n> 1$, for $n \ge 1$, will be specified later. 
  We define an auxiliary zoomed norm by
  $\Vert \varphi \Vert_n = \Vert \varphi \circ \ZZ_n^{-1} \Vert_{\HH^{\tau,p}}$.
  which is equivalent with $\Vert \cdot \Vert_{\HH^{\tau,p}}$,  for each fixed $n$.

  Let $\psi \colon \bR \to [0,1]$ be a $C^\infty$-function supported
  in $(-1,1)$ with $\sum_{m\in \bZ} \psi_m = 1$, where $\psi_m(x) =
  \psi(x+m)$. We define
  \[
  \psi_{m,n}(x) = \psi_m \circ \ZZ_n (x)= \psi_m (Z_n \cdot x) \, ,
  \quad m\in \bZ\, , n \in \bZ_+ \, .
  \]
Note that the support of $\psi_{m,n}$ becomes small if $Z_n$ is large.
For any $x \in I$ and any $n\ge 1$, there are
  at most three integers $m$ such that $\psi_{m,n}(x)\ne 0$,
and we may decompose
  \begin{align}\label{frag}
    \LL_t^n \varphi(x) = \sum_{m \in \bZ} \LL_t^n (\psi_{m,n} \varphi) (x)
    = \sum_{m \in \bZ} \sum_{J \in \cA(t, n)} 
\biggl (1_J \psi_{m,n} \frac{ \varphi}{ |(f^n_t)' | }\biggr )
 \circ  f^{-n}_{t, J} (x)\, ,
  \end{align}
where $\cA(t, n)$ is the
  partition of monotonicity associated to $f_t^n$,  while
   $f^{-n}_{t, J}$ is the extension to $I$ of the inverse of $f^{n}_t |_J$.

For any $t$, the intersection multiplicity 
\[
\sup_x \# \{\, (m ,J) \mid m \in \bZ\, , \, J \in \cA(t, n)\, , \, 
 x \in  f^n_t(J \cap \supp \psi_{m, n}) \,\}
\]
is bounded by $3\cdot  2^n$ (this is called the ``complexity
at the end'' in Thomine's paper).
 Therefore, by Lemma~3.5 in \cite{Thomine}
there is $C = C(\tau,p) >
  0$ and for each $n$ there exists $C_n= C(\tau,p, \psi,n)$ such that
  \begin{align}\nonumber
    \| \LL_t^n \varphi \|_{n}^p 
&\le C \sum_{m \in \bZ}   \sum_{J \in \cA(t, n)} 
 2^{n(p-1)} \biggl\| 
\biggl ( 1_J\psi_{m,n}  \frac{ \varphi}{ |(f^n_t)' | }\biggr )
 \circ  f^{-n}_{t, J}\biggr\|^p_{n} \\
\nonumber&\qquad\qquad\qquad+
C_n \sum_{m \in \bZ}   \sum_{J \in \cA(t, n)}  
 \biggl\| 
\biggl ( 1_J \psi_{m,n}  \frac{ \varphi}{ |(f^n_t)' | }\biggr )
 \circ  f^{-n}_{t, J} \circ \ZZ_n^{-1}   \biggr\|^p_{L^p}  \\ 
\label{one} &\le C \sum_{m \in \bZ} \sum_{J \in  \cA(t,n)} 
2^{n(p-1)} \biggl\| \biggl (1_J \psi_{m,n}  \frac{ \varphi}{ |(f^n_t)' | }\biggr )
 \circ  f^{-n}_{t, J}
  \biggr\|^p_{n} 
+ C_n Z_n \|  \varphi \|_{L^p}^p\, .
  \end{align}
  
  To control the double sum in  \eqref{one}, we write
  \begin{align*}
   (1_J \psi_{m,n} \varphi) \circ f^{-n}_{t, J} \circ \ZZ_n^{-1} = \tilde{\varphi}_{m,n,J} \circ \ZZ_n
    \circ f^{-n}_{t, J} \circ \ZZ_n^{-1}\, ,
  \end{align*}
  where $\tilde{\varphi}_{m,n,J} = (1_J \psi_{m,n} \varphi) \circ \ZZ_n^{-1}$. Since $\sup_t |f''_t|<\infty$  and $\inf _t \inf |f'_t|>1$, a
 standard bounded distortion\footnote{
We mean $|1-(f^n_t)'(\bar x)/(f^n_t)'(y)|\le K(f''_t, f'_t)\cdot  d(f^n_t(\bar x), f^n_t(y))$,
and the zoom $\ZZ_n$ ensures $K(f_t'', f_t')\cdot  d(f^n_t(\bar x), f^n_t(y))<1/2$.}  estimate gives for any $n$  a large enough $Z_n$ 
(uniformly in $t$, $m$, and $J$)
 such that, for each $t$ and $J\in \cA(t,n)$, 
and for any $y \in \text{supp}(1_J \psi_{m,n})$, we have,
  \begin{align*}
    \frac{1}{ 2| (f^n_t)' (y) |} \le
 \biggl |\frac{1}{(f^n_t)'  \circ f^{-n}_{t, J} \circ \ZZ_n^{-1} (x)}
    \biggr | 
\le \frac{3}{2| (f^n_t)'( y) |} \, , \, \forall x \in \supp((1_J \psi_{m,n} ) \circ f^{-n}_{t, J} \circ \ZZ_n^{-1})\, .
  \end{align*}
Applying
  Lemma~\ref{lem:compositionbysmooth} to 
  $\ZZ_n \circ f^{-n}_{t, J} \circ \ZZ_n^{-1}$, we obtain
 $C=C(\tau, p)$ independent of $n$ and $t$, and $C_n=C_n(p)$ independent of $t$
such that
\begin{align}\label{this}
  & \| (1_J \psi_{m,n} \varphi) \circ f^{-n}_{t, J} \circ
  \ZZ_n^{-1}\|_{\HH^{\tau,p}}\\ \nonumber &\qquad\qquad\le C |
  (f^n_t)' (y) |^{\frac1p} (\lambda_n(f_t))^{-\tau} \|
  \tilde{\varphi}_{m,n,J} \|_{\HH^{\tau,p}} + C_n \|
  \tilde{\varphi}_{m,n,J} \|_{L^p} \, , \, \, \forall |t| \le \ve
  \, .
\end{align}

Since $1/p<1$, the bound \eqref{this} together with
Lemma~\ref{lem:multiplicationbysmooth} for\footnote{Since the $C^1$ norm of $1/(f^n_t)'$ is bounded uniformly in $n$ and $t$
by some  $D$, choosing a suitable zoom $\ZZ_n$  allows to bound the $C^1$ norm of $(1_J\psi_{m,n}/(f^n_t)')\circ \ZZ_n^{-1}$
by $\lambda_n(f_t)^{-1}+D/Z_n\le C\lambda_n(f_t)^{-1}$.} $\gamma=1$ imply
\begin{align*}
  \biggl\|\biggl (1_J \psi_{m,n} \frac{ \varphi}{ |(f^n_t)'|} \biggr )
  \circ f^{-n}_{t, J} \biggr\|_{n} &\le C | (f^n_t)' (y) |^{\frac1p} (\lambda_n(f_t))^{-\tau-1} \| \tilde{\varphi}_{m,n,J}
  \Vert_{\HH^{\tau,p}} +C_n \| \tilde{\varphi}_{m,n,J} \|_{L^p}
  \\ &\le C (\lambda_n(f_t))^{-( \tau +1 - \frac1p ) }\|
  \tilde{\varphi}_{m,n,J} \|_{\HH^{\tau,p}} +C_n \|
  \tilde{\varphi}_{m,n,J} \|_{L^p} \, .
  \end{align*}

Therefore, by the convexity inequality $(|a|+|b|)^p\le
2^{p-1}(|a|^p+|b|^p)$, we have
\begin{multline}\label{two}
  \biggl\Vert \biggl (1_J \psi_{m,n} \frac{ \varphi}{ |(f^n_t)'|}
  \biggr ) \circ f^{-n}_{t, J} \biggr\Vert_{n}^p \\ \le C(p)
  (\lambda_n(f_t))^{-( p(\tau + 1) - 1 ) }\Vert
  \tilde{\varphi}_{m,n,J} \Vert_{\HH^{\tau,p}}^p + C_n \Vert
  \tilde{\varphi}_{m,n,J} \Vert_{L^p}^p \, ,
\end{multline}
where $C(p)$ and $C_n$ are independent of $t$.  By the Strichartz
Lemma~\ref{lem:Strichartz} (we use here $0\le \tau <1/p$) and the
(classical) localisation Theorem~2.4.7 in \cite{Triebel} (or Lemma~3.4
in \cite{Thomine} there exists $K = K(\psi, \tau,p) > 0$ such that for
all $m, n$ and $\varphi$
\begin{align}\nonumber
  \sum_{m \in \bZ} \| \tilde{\varphi}_{m,n,J} \|_{\HH^{\tau,p}}^p &=
  \sum_{m \in \bZ} \| (1_J \cdot (\psi_{m}\circ \ZZ_n) \cdot \varphi)
  \circ \ZZ_n^{-1} \|_{\HH^{\tau,p}}^p \\ & \le K \sum_{m \in \bZ} \|
  ( (\psi_{m}\circ \ZZ_n) \cdot \varphi) \circ \ZZ_n^{-1}
  \|_{\HH^{\tau,p}}^p
  \label{three}  \le K^2 \| \varphi \circ \ZZ_n^{-1}
  \|_{\HH^{\tau,p}}^p = K^2 \| \varphi \Vert_{n}^p\, .
\end{align}

For a fixed map $f_0$ and for any $n$, if $Z_n$ is large enough, then
for any $m$ there are at most two intervals $J\in \cA(0,n)$ such that
$(1_J \psi_{m,n}) \circ f^{-n}_J$ is not identically zero. (This is
the ``complexity at the beginning'' in Thomine's argument.)  For a
perturbation $(f_t)$, caution is needed, especially (but not only) if
$c$ is periodic for $f$ (see footnote \ref{goodness}).

Assume first that $c$ is not periodic.  Fix $\tilde \Lambda < \bar
\Lambda < \Lambda$, choose $p_0\in (1,p(f))$ and $n_0 \ge 1$ such that
\begin{equation}\label{startt}
  2^{p_0-1} \cdot (\lambda_{n_0}(f_t))^{- p_0(\tau + 1) - 1 } \le \bar
  \Lambda^{n_0(- p_0(\tau + 1) - 1 )} \, , \,\, \forall |t| <\ve_0 \,
  .
\end{equation}
Next, for $p \in (1,p_0)$, let $C(p)$ be the constant from
\eqref{two}, let $K(p)$ be the constant from \eqref{three}. Then,
there exists $N\ge n_0$ such that
\begin{equation}\label{largeN}
  2 \cdot C(p) \cdot K(p)^2 \cdot 2^{N(p-1)} \cdot
  (\lambda_N(f_t))^{-( p(\tau + 1) - 1 ) } < \tilde \Lambda ^{-N(
    p(\tau + 1) - 1 ) } \, .
\end{equation}
Now, since $c$ is not periodic, (see e.g.\ \cite[p.~376]{BY}) there
exists $\ve'=\ve '(N)<\ve_0$ such that for any $|t|\le \ve'(N)$ there
is a natural bijection $\JJ_t=\JJ_{t,N}$ between $\cA(0,N)$ and
$\cA(t,N)$. In addition, we may ensure that $ |f^N_t (\JJ_t(J))|
\ge{|f^N(J)|}/{2} $ for all $J\in \cA(0,N)$ and all $|t|\le \ve'(N)$.
By bounded distortion, there exists $Z_N$ such that for all $|t|\le
\ve'(N)$ and any $m$ there are at most two intervals $J \in \cA(t,N)$
such that $(1_J \psi_{m,N} \varphi) \circ f^{-N}_{t, J}$ is not
identically zero.  Therefore, it follows from \eqref{two} and
\eqref{three} that
\begin{align}
 \nonumber & \sum_{m \in \bZ} \sum_{J \in \cA(t,N)} \biggl\|
 \frac{(1_J \psi_{m,N} \varphi) \circ f^{-N}_{t, J} }{ (f^N_t)'
   \circ f^{-N}_{t, J} } \biggr\|_{N}^p \\ &\qquad \le \sum_{m
   \in \bZ} \sum_{J \in \cA(t,N)} \biggl( C (\lambda_N(f_t))^{-(
   p(\tau + 1) - 1 ) }\| \tilde{\varphi}_{m,N,J}
 \|_{\HH^{\tau,p}}^p + C_N \| \tilde{\varphi}_{m,N,J} \|_{L^p}^p
 \biggr) \\ &\qquad \le C K^2 (\lambda_N(f_t))^{-( p(\tau
   + 1) - 1 ) } \| \varphi \|_{N}^p + 2 C_N \| \varphi \|_{L^p}^p
 \, , \quad \forall |t|\le \ve'(N) \, .
\end{align}
Recalling \eqref{one}, we have shown that for all $|t|\le \ve'(N)$,
\begin{align*}
  \Vert \LL_t^N \varphi \Vert_{N}^p \le C K^2 2^{N(p-1)}
  (\lambda_N(f_t))^{-( p(\tau + 1) - 1 ) }\Vert \varphi \|_{N}^p + 2
  C_N \| \varphi \|_{L^p}^p \, .
\end{align*}
Finally, we use \eqref{eq:sobolev_scaling} twice to obtain 
\begin{align}\label{final}
  \Vert \LL_t^N \varphi \Vert_{\HH^{\tau,p}} \le 2 C K^2
  2^{N(1-\frac1p)} (\lambda_N(f_t))^{-( \tau + 1 - \frac1p ) }\|
  \varphi \Vert_{\HH^{\tau,p}} + 2 ( C_N + C_\mathrm{zoom}) \Vert
  \varphi \Vert_{L^p} \, ,
\end{align}
for all $|t|\le \ve'(N)$, which together with our choice
\eqref{largeN} of $N$ and the Sobolev embedding theorem
completes the proof of
\eqref{eq:lasota_yorke'} if $c$ is not periodic, for $\ve_0:=\ve'(N)$.

\smallskip
If $c$ is periodic but all the $f_t$ are topologically conjugated to
$f_0$, the argument in the non periodic case can be applied.

\smallskip
Otherwise, if $c$ is periodic of minimal period $P_f \ge 2$, 
  we first take $\tilde \Lambda < \bar \Lambda < \Lambda$ and
  choose $p_0 \in (1, p(f))$ and $n_0=k_0 P_f$ so that 
  \begin{equation}\label{start2}
  2^{k_0} 2^{p_0-1} \cdot  (\lambda_{n_0}(f_t))^{- p_0(\tau + 1) - 1 } \le \bar \Lambda^{n_0(-
  	p_0(\tau + 1) - 1 )} \, , \, \, \forall |t|< \ve _0 \, .
  \end{equation}
Then we take $p \in (1, p_0)$, and we fix $N=k_1P_f \ge n_0$ large enough
such that 
\begin{equation}\label{largeN'}
  2^{k_1}   \cdot C(p) \cdot K(p)^2   \cdot 2^{N(p-1)} 
 \cdot (\lambda_N(f_t))^{- p(\tau + 1) - 1  }
< \tilde \Lambda ^{-N( p(\tau + 1) - 1 ) } \, , \,\,\forall |t|< \ve _0  \, .
\end{equation} 
Now, (\cite[p.~376]{BY}), for $N=k_1 P_f$, there exists 
$\ve''=\ve ''(N)<\ve_0$ such that for any $|t|\le \ve''(N)$,
there is a natural injection
$\JJ_t=\JJ_{t,N}$
from $\cA(0,N)$ to $\cA(t,N)$ such that
$
|f^N_t (\JJ_t(J))| \ge{|f^N(J)|}/{2} 
$ for all $J\in \cA(0,N)$ and all $|t|\le \ve'(N)$, and,
in addition, between any two intervals of $\cA(t,N)$
 in the image of $\JJ_{t,N}$, there are
at most $2^{k_1}-2$  intervals of $\cA(t,N)$ which are \emph{not}
 in the image of $\JJ_{t,N}$.
Therefore, there exists $Z_N$ (uniform in $t$) such that for any $|t|\le \ve''(N)$
 and any $m \in \bZ$, there are at most $2^{k_1}$
  intervals $J \in \cA(t,N)$ such that $(1_J \psi_{m,N} \varphi) \circ f^{-N}_{t, J}$ is
  not identically zero. Therefore, \eqref{final} is replaced by 
\begin{align*}
    \| \LL_t^{N} \varphi \|_{\HH^{\tau,p}} \le 2^{k_1} [ C K^2
      2^{N(1-\frac1p)} (\lambda_N(f_t))^{-( \tau + 1 - \frac1p ) }] \|
    \varphi \|_{\HH^{\tau,p}} +2 (C_{N} + C_\mathrm{zoom}) \| \varphi
    \|_{L^p} \, ,
  \end{align*}
for $N=k_1 P_f$ and  all $|t|<\ve''(N)$.
Recalling \eqref{defLe} and \eqref{largeN'}, we may conclude  for
 $\ve_0:=\ve''(N)$.  

\end{document}